\documentclass[amstex,12pt]{amsart}
\usepackage{amssymb}
\usepackage{amsmath}
\usepackage{amscd}
\usepackage{mathrsfs}
\usepackage{graphicx}
\usepackage{latexsym}

\theoremstyle{definition} 

\newtheorem*{definition}{Definition} 
\newtheorem*{remark}{Remark}

\theoremstyle{plain}  

\newtheorem{theorem}{Theorem}[section]
\newtheorem{proposition}[theorem]{Proposition}
\newtheorem{corollary}[theorem]{Corollary}
\newtheorem{lemma}[theorem]{Lemma}

\DeclareMathOperator{\Isom}{Isom}

\DeclareMathOperator{\id}{id}

\DeclareMathOperator{\esssup}{esssup}
\DeclareMathOperator{\QC}{QC}
\DeclareMathOperator{\AC}{AC}
\DeclareMathOperator{\QS}{QS}

\DeclareMathOperator{\Sym}{Sym}
\DeclareMathOperator{\Mob}{\mbox{\rm{M\"ob}}}
\DeclareMathOperator{\Diff}{Diff}

\DeclareMathOperator{\Hol}{Hol}
\DeclareMathOperator{\Bel}{Bel}
\DeclareMathOperator{\Ael}{Ael}

\DeclareMathOperator{\D}{\mathbb D}

\DeclareMathOperator{\R}{\mathbb R}
\DeclareMathOperator{\N}{\mathbb N}
\DeclareMathOperator{\Z}{\mathbb Z}
\DeclareMathOperator{\H2}{\mathbb H}
\DeclareMathOperator{\S1}{\mathbb S}
\DeclareMathOperator{\Chat}{\widehat {\mathbb C}}

\begin{document}

\title[Circle diffeomorphisms and Teich\-m\"ul\-ler spaces]{Rigidity of groups of circle diffeomorphisms \\and Teich\-m\"ul\-ler spaces}

\author{Katsuhiko Matsuzaki}
\address{Department of Mathematics, School of Education, Waseda University,\endgraf
Shinjuku, Tokyo 169-8050, Japan}
\email{matsuzak@waseda.jp}

\subjclass[2010]{Primary 30F60, Secondary 37E30}
\keywords{quasiconformal map, Schwarzian derivative, Bers embedding, quasisymmetric homeomorphism, H\"older continuous derivative, integrable Teich\-m\"ul\-ler space, Weil--Petersson metric}
\thanks{This work was supported by JSPS KAKENHI 25287021.}

\begin{abstract}
We consider deformations of a group of 
circle diffeomorphisms with H\"older continuous derivative in the framework of 
quasiconformal Teich\-m\"ul\-ler theory
and show certain rigidity under conjugation
by symmetric homeomorphisms of the circle. As an application, we give 
a condition for such a diffeomorphism group to be conjugate to 
a M\"obius group by a diffeomorphism of the same regularity.
The strategy is to find a fixed point of the group which acts isometrically on
the integrable Teich\-m\"ul\-ler space with the Weil--Petersson metric.
\end{abstract}

\maketitle

\section
{Introduction and the statement of theorems}\label{0}
In this paper, we prove certain rigidity of the deformation of a Fuchsian group within
a group of circle diffeomorphisms. The regularity of the diffeomorphisms we consider here is
such that their derivatives are H\"older continuous. This class is important not only in the 
theory of one dimensional dynamics 
but also 
for certain problems on the smoothness of foliations of codimension one in closed 3-manifolds (see e.g. \cite{Gh}, \cite{HK}). 
For a constant $\alpha \in (0,1)$,
we denote by $\Diff_+^{1+\alpha}(\S1)$ the group of all orientation-preserving diffeomorphisms $g$ of 
the unit circle $\S1$ whose derivatives are $\alpha$-H\"older continuous. This means that
there is a constant $c=c(g) \geq 0$ such that
$$
|g'(x)-g'(y)| \leq c|x-y|^\alpha
$$
for any $x,y \in \S1=\R/2\pi \mathbb Z$. 

We formulate rigidity phenomena of a group of circle diffeomorphisms
in the framework of quasiconformal Teich\-m\"ul\-ler theory. 
The universal Teichm\"uller space $T$
is identified with the group $\QS$ of
quasisymmetric self-homeomorphisms $g$ of $\S1$ modulo the group $\Mob(\S1)$ of M\"obius transformations of $\S1$,
i.e., $T=\Mob(\S1) \backslash \QS$.
A quasisymmetric homeomorphism $g$ is the boundary extension of a quasiconformal self-homeomorphism $\widetilde g$ of
the unit disk $\D$. The group of all such $\widetilde g$ is denoted by $\QC(\D)$. By the solution of the Beltrami equation,
$\QC(\D)$ modulo the group $\Mob(\D)$ of M\"obius transformations of $\D$
is identified with the space of Beltrami coefficients on $\D$ denoted by
$$
\Bel(\D)=\{\mu \in L^\infty(\D) \mid \Vert \mu \Vert_\infty<1\}.
$$
Thus, we can regard $T$ as the space of
equivalence classes of $\Bel(\D)$, where the equivalence relation
is given by the coincidence of the boundary extension of quasiconformal homeomorphisms.
The quotient map $\pi:\Bel(\D) \to T$ is called the Teichm\"uller projection.
See Section 2 for further details.

We introduce the Teichm\"uller space $T_0^\alpha$ of $\Diff^{1+\alpha}_+(\S1)$ as a subspace of
the universal Teichm\"uller space $T$. This is defined by 
$T_0^\alpha=\Mob(\S1) \backslash \Diff^{1+\alpha}_+(\S1)$.
If we prepare a subspace of $\Bel(\D)$ to be
$$
\Bel_0^\alpha(\D)=\{\mu \in \Bel(\D) \mid \esssup_{|z|>1-t} |\mu(z)|=O(t^\alpha) \ (t \to 0)\},
$$
then $T_0^\alpha$ is its image under the Teichm\"uller projection; $T_0^\alpha=\pi(\Bel_0^\alpha(\D))$.
After Carleson \cite{Car}, many authors contributed to the problems of relationship between those spaces involved in $\Diff^{1+\alpha}_+(\S1)$.
Moreover, we can put them in the framework of the theory of the universal Teichm\"uller space (\cite{Mat2}).
We review these results in Section 3.

A purpose of this paper is to consider the following conjugation problem in $\Diff^{1+\alpha}_+(\S1)$:
Given $G \subset \Diff^{1+\alpha}_+(\S1)$, find some $f \in \Diff^{1+\alpha}_+(\S1)$ such that
$fGf^{-1} \subset \Mob(\S1)$.
This problem can be interpreted as 
the fixed point problem of group action on a Teichm\"uller space as follows.
The quasisymmetric group   
$\QS$ acts on $T=\Mob(\S1)\backslash \QS$ canonically: the representation
$T \times \QS \to T$ is given by $([f],g) \mapsto [f \circ g]=:g^*[f]$.
For a subgroup $G \subset \QS$ and an element $f \in \QS$, we see that the condition
$fGf^{-1} \subset \Mob(\S1)$ is equivalent to the condition $g^*[f]=[f]$ for every $g \in G$.
Similarly, $\Diff^{1+\alpha}_+(\S1)$ acts on $T_0^\alpha=\Mob(\S1) \backslash \Diff^{1+\alpha}_+(\S1)$.
Hence, for $G \subset \Diff^{1+\alpha}_+(\S1)$, 
to find the conjugation of $G$ into $\Mob(\S1)$ by $f \in \Diff^{1+\alpha}_+(\S1)$ is to
find a fixed point $[f]$ of $G$ in $T_0^\alpha$.
However, $T_0^\alpha$ is not an appropriate space to find a fixed point of the group action.
We utilize a larger space than $T_0^\alpha$.

We consider the space of $p$-integrable Beltrami coefficients  
$$ 
\Ael^p(\D)=\{ \mu \in \Bel (\D) \mid \Vert \mu \Vert_p<\infty\} \quad(p \geq 2),
$$
where the norm is given by
$$
\Vert \mu \Vert^p_p=\int_{\D} |\mu(z)|^p \rho_{\D}^2(z) dxdy \
$$
for the hyperbolic density $\rho_{\D}(z)=2/(1-|z|^2)$ on $\D$.
The $2$-integrable Teichm\"uller space is defined by its image $T^2=\pi(\Ael^2(\D))$ under the
Teichm\"uller projection.
This space and the Weil--Petersson metric $d_{WP}^2$ on it were first introduced by Cui \cite{Cui}. 
It is easy to see that
if $\alpha>1/2$, then $\Bel_0^\alpha(\D) \subset \Ael^2(\D)$, and hence 
$T_0^\alpha \subset T^2$.

The negatively curved property of $(T^2,d_{WP}^2)$
was proved by Takhtajan and Teo \cite{TT}. 
In particular, this is a ${\rm CAT}(0)$ space as a metric space. Moreover,
$\Diff^{1+\alpha}_+(\S1)$ acts on $(T^2,d_{WP}^2)$ isometrically.
As a fixed point property of isometric action on a ${\rm CAT}(0)$ space $X$ in general, it is well-known that
$G \subset \Isom(X)$ has a fixed point in $X$ if and only if any orbit of $G$ is a bounded subset of $X$. 
See \cite[Chapter II.2]{BH}.

In the situation of our conjugation problem under the assumption $\alpha>1/2$,
we consider the isometric action of 
a subgroup $G \subset \Diff^{1+\alpha}_+(\S1)$ on $(T^2,d_{WP}^2)$.
Then, by describing the boundedness of the orbit of $G$ in $(T^2,d_{WP}^2)$ in some explicit way,
we will have a condition for $G$ to have a fixed point in $T^2$.
However, there is one point missed for our desired result; we want to find a fixed point in the smaller space $T_0^\alpha$.
The final piece of our argument is filled by the following rigidity theorem proved in Section \ref{4}.

\newtheorem{first}{Theorem}
\renewcommand{\thefirst}{\ref{th2}}
\begin{first}[rigidity]\label{20}
Let $\Gamma$ be a subgroup of $\Mob(\S1)$ that contains a hyperbolic element.
If $f \Gamma f^{-1} \subset \Diff^{1+\alpha}_+(\S1)$ for $f \in \Sym$, then $f \in \Diff^{1+\alpha}_+(\S1)$.
\end{first}

Here, $g \in \QS$ is called symmetric if some quasiconformal extension $\widetilde g \in \QC(\D)$ of $g$
is an asymptotically conformal homeomorphism of $\D$
whose Beltrami coefficient belongs to
$$
\Bel_0(\D):=\{\mu \in \Bel(\D) \mid \esssup_{|z|>1-t} |\mu(z)| \to 0\ (t \to 0)\}.
$$ 
The subgroup of $\QS$ consisting of all symmetric self-homeomorphisms of $\S1$ is denoted by 
$\Sym$. The corresponding Teichm\"uller space is defined by
$T_0=\Mob(\S1) \backslash \Sym=\pi(\Bel_0(\D))$,
which was introduced by Gardiner and Sullivan \cite{GS}.
This is also explained in Section 2.
We note that the inclusion relations $T_0^\alpha \subset T^2 \subset T_0$ are satisfied.
The rigidity theorem shows that
the conjugation between $G \subset \Diff^{1+\alpha}_+(\S1)$ and $\Gamma \subset \Mob$ by  
$f \in \Sym$ is given as an inner automorphism of $\Diff^{1+\alpha}_+(\S1)$.
In other words, 
if $G \subset \Diff^{1+\alpha}_+(\S1)$ has
a fixed point in $T_0$,
then it is in $T_0^\alpha$. 

In Section 5, we state our main result on the conjugation problem as follows.

\newtheorem{second}{Theorem}
\renewcommand{\thesecond}{\ref{conjugate}}
\begin{second}[trivial conjugation]\label{conjugate0}
Let $G$ be an infinite non-abelian subgroup of $\Diff_+^{1+\alpha}(\mathbb S)$ with $\alpha \in (1/2,1)$.
Then, the following conditions are equivalent:
\begin{enumerate}
\item
There exists some $f \in \Diff_+^{1+\alpha}(\mathbb S)$ such that
$f G f^{-1} \subset \Mob(\mathbb S)$;
\item
The orbit of $G$ is bounded in $T^2$
with respect to $d_{WP}^2$;
\item
There exist positive constants 
$\kappa_2 <\infty$ and $\kappa_\infty<1$
such that
$$
{\rm (i)}\ \inf_{\pi(\mu)=[g]} \Vert \mu \Vert_2 \leq \kappa_2; \quad 
{\rm (ii)}\ \inf_{\pi(\mu)=[g]} \Vert \mu \Vert_\infty \leq \kappa_\infty
$$
for all $g \in G$.
\end{enumerate}
Moreover, if $(1)$ holds, then such an $f \in \Diff^{1+\alpha}_+(\S1)$ is unique up to the post-composition of a
M\"obius transformation.
\end{second}

The equivalence $(1) \Leftrightarrow (2)$ can be seen by the previous arguments.
A topology on $T^2$ induced by $\Vert \cdot \Vert_\infty+\Vert \cdot \Vert_2$ on $\Ael^2(\D)$
coincides with the topology defined by the Weil--Petersson metric (\cite{Cui}).
The implication $(3) \Rightarrow (2)$ means that
the boundedness with respect to the Weil--Petersson metric can be also
detected by the norms $\Vert \cdot \Vert_\infty$ and $\Vert \cdot \Vert_2$. 
This property is formulated in Theorem \ref{WPdistance}.
The uniqueness is due to another rigidity theorem (Theorem \ref{th1}).

A more restricted sufficient condition
for the conjugation will be also obtained without the assumption $\alpha >1/2$ (Theorem \ref{conjugate2}).
To prove this, we use the $p$-integrable Teich\-m\"ul\-ler space $T^p$ defined by 
$\pi(\Ael^p(\D))$ 
for $\alpha>1/p$. 
The proofs of these theorems in Section 5
will be given in Sections \ref{6} and \ref{7}.

Results obtained in this paper and 
the related work \cite{Mat2} have been announced in survey articles \cite{Mat0} and \cite{Mat3}.
As another application of the rigidity theorem, we can embed the deformation space of a Fuchsian group $\Gamma$
within $\Diff^{r}_+(\S1)$ for any $r>1$ into the deformation space $AT(\Gamma)$ of $\Gamma$ in $\Sym$. 
This is also sketched out in \cite{Mat3}; the detail is contained in \cite{Mat7}.
The space $AT(\Gamma)$ has been studied in \cite{Mat6} as the Teich\-m\"ul\-ler space of $\Gamma$-invariant symmetric structures 
on $\S1$.

\section{The universal Teich\-m\"ul\-ler space and its little subspace}\label{1}
In this section, we review basic facts on the universal Teich\-m\"ul\-ler space and its little subspace.
The reader may refer to monographs by Lehto \cite{Leh} and Nag \cite{Nag} for quasiconformal aspects of 
Teich\-m\"ul\-ler spaces.

An orientation-preserving homeomorphism $w$ of a domain in the complex plane
is said to be {\it quasiconformal} if partial derivatives $\partial w$ and $\bar \partial w$ in
the distribution sense exist and
if the complex dilatation
$\mu_w(z)=\bar \partial w(z)/\partial w(z)$ satisfies $\Vert \mu_w \Vert_\infty <1$.
Let $\Bel(\D)$ 
be the space of measurable functions $\mu$ on the unit disk $\D$ with $\Vert \mu \Vert_\infty<1$, 
which are called {\it Beltrami coefficients}.
We set the group of all quasiconformal self-homeomorphisms of $\D$ by $\QC(\D)$.
By the measurable Riemann mapping theorem (see \cite{Ah0}), for every $\mu \in \Bel(\D)$, there is $w \in \QC(\D)$ satisfying $\mu_w=\mu$
uniquely up to the post-composition of elements of $\Mob(\D) \cong {\rm PSL}(2,\mathbb R)$, the group of 
all M\"obius transformations of $\D$. This gives the identification
$\Mob(\D)\backslash \QC(\D) \cong \Bel(\D)$. 

Every $w \in \QC(\D)$ extends continuously to a quasisymmetric homeomorphism of $\S1=\partial \D$.
Here an orientation-preserving self-homeomorphism $g:\S1 \to \S1$ is called
{\it quasisymmetric} if there is a constant $M \geq 1$ such that
the quasisymmetry quotient $m_g(x,t)$ satisfies 
$$
\frac{1}{M} \leq m_g(x,t):=\frac{g(x+t)-g(x)}{g(x)-g(x-t)} \leq M
$$
for every $x \in \S1=\mathbb R/2\pi \mathbb Z$ and for every $t>0$. 
Here, $g$ can be identified with its lift $\R \to \R$.
Let $\QS$ be the group of all quasisymmetric self-homeomorphism of $\S1$. 
We denote the boundary extension map by
$q:\QC(\D) \to \QS$,
which is known to be a surjective homomorphism.
The {\it universal Teich\-m\"ul\-ler space} is defined by
$T=\Mob(\S1) \backslash \QS$. 
Then, the boundary extension $q$ induces the Teich\-m\"ul\-ler projection $\pi:\Bel(\D) \to T$.

The quotient topology of $T$ is induced from $\Bel(\D)$ by $\pi$. In fact, the Teich\-m\"ul\-ler distance $d_T$
can be defined by
$$
d_T(\tau_1,\tau_2)=\inf_{\substack{\pi(\mu_1)=\tau_1\\ \pi(\mu_2)=\tau_2}} 
\log \frac{1+\left\Vert\frac{\mu_1-\mu_2}{1-\bar \mu_2 \mu_1}
\right\Vert_\infty}{1-\left\Vert\frac{\mu_1-\mu_2}{1-\bar \mu_2 \mu_1}\right\Vert_\infty}
$$ 
for any $\tau_1,\tau_2 \in T$, where the infimum is taken over all $\mu_1, \mu_2 \in \Bel(\D)$ with
$\pi(\mu_1)=\tau_1$ and $\pi(\mu_2)=\tau_2$.

The group $\QS$ acts on $T$ from the right canonically: for $[f] \in T$ and $g \in \QS$, we define
$g^*[f]:=[f \circ g] \in T$.
This is regarded as the mapping class group of
the universal Teich\-m\"ul\-ler space. The action is faithful and transitive. Moreover,
this is isometric with respect to $d_T$. 
The isotropy subgroup of $\QS$ at the origin $o=[\id] \in T$
coincides with $\Mob(\S1)$.
The condition that
$g \in \QS$ fixes $[f] \in T$, that is,
$g^*[f]=[f]$, can be written as $[fgf^{-1}]=[\id]$, and this is equivalent to 
the condition that $fgf^{-1} \in \Mob(\S1)$.

For any $\mu \in \Bel(\D)$, we
extend it to a Beltrami coefficient $\widehat \mu$ on the Riemann sphere $\widehat{\mathbb C}$
by setting $\widehat \mu(z) \equiv 0$ for $z \in \mathbb D^*=\widehat{\mathbb C}-\overline{\mathbb D}$.
We denote a quasiconformal homeomorphism of $\Chat$ with complex dilatation $\widehat \mu$ by
$f_\mu$. We remark here that quasiconformality near $\infty$ can be defined by that near $0$ under the conjugation of $1/z$.
The measurable Riemann mapping theorem guarantees the existence of such $f_\mu$
and the uniqueness of $f_\mu$ up to the post-composition of M\"obius transformations of $\widehat{\mathbb C}$.

We take the Schwarzian derivative 
$S_{f_\mu}:\mathbb D^* \to \widehat{\mathbb C}$ of 
the conformal homeomorphism $f_\mu|_{\mathbb D^*}$. 
This measures the difference of the marked complex projective structure on $\D^*$ from the standard one.
By the Nehari-Kraus theorem, 
$S_{f_\mu}$ belongs to the following complex Banach space of holomorphic quadratic differentials on $\D^*$:
$$
B(\D^*)=\{\varphi \in \Hol_2(\D^*) \mid \Vert \varphi \Vert_\infty =\sup_{z \in \D^*} \rho^{-2}_{\D^*}(z)|\varphi(z)|<\infty\},
$$
where $\rho_{\D^*}(z)=2/(|z|^2-1)$ is the hyperbolic density on $\D^*$.
We note that $\varphi \in \Hol_2(\D^*)$ is a holomorphic function on $\{|z|>1\}$ with $\varphi(z)=O(z^{-4})$ $(z \to \infty)$.
By this correspondence $\mu \mapsto S_{f_\mu}$, a holomorphic map
$$
\Phi:\Bel(\D) \to B(\D^*)
$$
is defined, which is called the {\it Bers projection} (onto the image).

For the Teich\-m\"ul\-ler projection $\pi:\Bel(\D) \to T$ and
the Bers projection $\Phi:\Bel(\D) \to B(\D^*)$, we can show that
$\Phi \circ \pi^{-1}$ is well-defined and injective, which defines the
{\it Bers embedding} $\beta:T \to B(\D^*)$.
In fact, $\beta$ is a homeomorphism onto the image $\beta(T)=\Phi(\Bel(\D))$ and $\beta(T)$ is a bounded domain in $B(\D^*)$.
This provides a complex Banach manifold structure for $T$ under which $\beta$ is biholomorphic. 

Every element $\gamma \in \Mob(\S1)$ acts on $B(\D^*)$ linear isometrically through the Bers embedding $\beta$.
This means that, for any point $[f] \in T$ with $\beta([f])=\varphi \in \beta(T)$, the Bers embedding $\beta(\gamma^*[f])$ of 
the image $\gamma^*[f]=[f \circ \gamma]$
is represented by
$$
(\gamma^*\varphi)(z)=\varphi(\gamma(z))\gamma'(z)^2,
$$
where we regard $\gamma$ as the element of $\Mob(\D^*)$ and 
$\gamma^*\varphi$ is the pull-back of $\varphi$ as a quadratic differential form.
Clearly this action extends to $B(\D^*)$ and satisfies $\Vert \gamma^*\varphi \Vert_\infty=\Vert \varphi \Vert_\infty$.
More generally, if $g \in \QS$ has a fixed point $[f] \in T$, then the action of $g$ on $\beta(T)$ is conjugate to
the linear isometric action of $fgf^{-1}$ on $\beta(T) \subset B(\D)$ under the base point change automorphism
$R_{[f]}:T \to T$ given by $[g] \mapsto [g \circ f^{-1}]$.

A quasiconformal homeomorphism $w \in \QC(\D)$ is called {\it asymptotically conformal} if 
the complex dilatation vanishes at the boundary, that is, $\mu_w(z) \to 0 \ (|z| \to 1)$ essentially uniformly.
The subspace of $\Bel(\D)$ consisting of all Beltrami coefficients vanishing at the boundary is denoted by $\Bel_0(\D)$ and
the subgroup of $\QC(\D)$ consisting of all asymptotically conformal self-homeomorphisms of $\D$ is denoted by
$\AC(\D)$. Moreover, a quasisymmetric homeomorphism $g \in \QS$ is called
{\it symmetric} if
the quasisymmetry quotient
$m_g(x,t)$ tends to $1$ as $t \to 0$ uniformly with respect to $x \in \S1$.
We note that a circle diffeomorphism is symmetric, but a symmetric homeomorphism is not necessarily 
absolutely continuous.
The group of all symmetric self-homeomorphisms of $\S1$ is denoted by
$\Sym$. Then, the restriction of the boundary extension to $\AC(\D)$ gives a surjective homomorphism
$q:\AC(\D) \to \Sym$. 

Gardiner and Sullivan \cite{GS} studied the asymptotic Teich\-m\"ul\-ler space defined by
$AT=\Sym \backslash \QS$,
and the {\it little universal Teich\-m\"ul\-ler space} defined by
$$
T_0=\Mob(\S1) \backslash \Sym=\pi(\Bel_0(\D)). 
$$
They introduced $\Sym$ as the characteristic topological subgroup of $\QS$ consisting of all elements $g \in \QS$
such that the adjoint map $\QS \to \QS$ given by conjugation of $g$ is continuous at the identity.

Here,
we summarize the characterization of symmetric homeomorphisms in $\Sym$,
Beltrami coefficients vanishing at the boundary in $\Bel_0(\D)$
and the Bers embedding of the little universal Teich\-m\"ul\-ler space $T_0$. 
We set a Banach subspace of $B(\D^*)$ consisting of all elements vanishing at the boundary by
$$
B_0(\D^*)=\{\varphi \in B(\D^*) \mid \lim_{|z| \to 1} \rho^{-2}_{\D^*}(z)|\varphi(z)|=0\}.
$$
The following result appeared in \cite{GS}, 
which were attributed to Fehlmann \cite{F} and Becker and Pommerenke \cite{BP}.
Condition (4) was also in the latter paper.

\begin{proposition}\label{gs}
For a quasisymmetric homeomorphism $g \in \QS$, the following conditions are equivalent:
\begin{enumerate}
\item
$g$ belongs to $\Sym$;
\item
there exists $\mu \in \Bel_0(\D)$ such that $\pi(\mu)=[g] \in T$;
\item
$\beta([g]) \in \beta(T)$ is in $B_0(\D^*)$;
\item
$\lim_{|z| \to 1+}(|z|-1)\left |f_\mu''(z)/f_\mu'(z) \right|=0$ for any $\mu \in \Bel(\D)$ with $\pi(\mu)=[g]$.
\end{enumerate}
\end{proposition}

The proof of our rigidity theorem is carried out by using the Bers embedding of the Teich\-m\"ul\-ler space; we 
transfer the problems to those for
the linear isometric action of M\"obius transformations on $B(\D^*)$.
Here, we apply the Bers embedding to $T_0$ and
prove a prototype of the rigidity theorem as follows.
The argument appeared earlier in \cite[Theorem 1]{Mat1}.

\begin{theorem}[little rigidity]\label{th1}
Let $\Gamma$ be a subgroup of $\Mob(\S1)$ that contains a hyperbolic or parabolic element.
If $f \Gamma f^{-1} \subset \Mob(\S1)$ for $f \in \Sym$, then $f \in \Mob(\S1)$.
\end{theorem}

\begin{proof}
Let $\varphi=\beta([f])$, which belongs to the subspace $B_0(\D^*)$ by Proposition \ref{gs}.
The condition $f \Gamma f^{-1} \subset \Mob(\S1)$ is equivalent to that
$\gamma^*\varphi=\varphi$
for every $\gamma \in \Gamma \subset \Mob(\D^*)$. 
Then,
$$
\rho_{\D^*}^{-2}(z)|\varphi(z)|=\rho_{\D^*}^{-2}(z)|\gamma^*\varphi(z)|
=\rho_{\D^*}^{-2}(\gamma z)|\varphi(\gamma z)|.
$$
Because $\varphi \in B_0(\D^*)$ and there is a sequence $\gamma_n \in \Gamma$ (given by
iterations of a hyperbolic or parabolic element) such that $|\gamma_n(z)| \to 1$
$(n \to \infty)$ for all $z \in \D^*$,
we have $\varphi(z) \equiv 0$. This means that
$[f]=[\id]$, and equivalently $f \in \Mob(\S1)$.
\end{proof}

\section{The Teich\-m\"ul\-ler space of circle diffeomorphisms}

As Carleson \cite{Car} first clarified, 
the decay order of the complex dilatation $\mu_w(z)$ of 
an asymptotically conformal homeomorphism $w \in \AC(\D)$
as $|z| \to 1$ reflects
that of $m_g(x,t)-1$ for its boundary extensions $g=q(w) \in \Sym$
as $t \to 0$. 
Moreover, the decay order of $m_g(x,t)-1$ is related to 
the exponent $\alpha$ of the H\"older continuity of the derivative of $g$.
We can bring these properties 
to the theory of Teich\-m\"ul\-ler spaces.

First, we give the characterization of $\Diff^{1+\alpha}_+(\S1)$ for $\alpha \in (0,1)$
analogously to Proposition \ref{gs}. Here are spaces we consider:
\begin{align*}
& T^\alpha_0=\Mob(\S1) \backslash \Diff^{1+\alpha}_+(\S1);\\
& \Bel^{\alpha}_0(\D)=\{\mu \in \Bel_0(\D) \mid \Vert \mu \Vert_{\infty,\alpha}=
\esssup_{z \in \D} \rho^{\alpha}_{\D}(z)|\mu(z)|<\infty\};\\
& B^{{ \alpha}}_0(\D^*)
=\{\varphi \in B_0(\D^*) \mid \Vert \varphi \Vert_{\infty,\alpha}=\sup_{z \in \mathbb D^*} 
\rho^{-2+{ \alpha}}_{\D^*}(z)|\varphi(z)|<\infty \}.
\end{align*}
As usual $\rho_{\D}(z)$ and $\rho_{\D^*}(z)$ are the hyperbolic densities on $\D$ and $\D^*$.
We regard $T^\alpha_0$ as the {\it Teich\-m\"ul\-ler space of circle diffeomorphisms} of $\alpha$-H\"older continuous derivatives.

\begin{theorem}\label{main}
For a quasisymmetric homeomorphism $g \in \QS$, the following conditions are equivalent:
\begin{enumerate}
\item
$g$ belongs to $\Diff^{1+\alpha}_+(\S1)$;
\item
there exists $\mu \in \Bel_0^\alpha(\D)$ such that $\pi(\mu)=[g] \in T$;
\item
$\beta([g]) \in \beta(T)$ is in $B_0^\alpha(\D^*)$;
\item
$\sup_{z \in \D^*}(|z|-1)^{1-\alpha}\left |f_\mu''(z)/f_\mu'(z) \right|<\infty$ 
for any $\mu \in \Bel(\D)$ with $\pi(\mu)=[g]$.
\end{enumerate}
\end{theorem}

These results are summarized in \cite[Theorems 4.1, 4.6 and 6.7]{Mat2}, 
based on the work of
Carleson \cite{Car} and its refinement by many authors. 
Moreover, we can estimate the relevant quantity in each item above in terms of other quantities.
In particular, the estimate of the norm of $B_0^\alpha(\D^*)$ by the norm of $\Bel_0^\alpha(\D)$ is
important in our study.

In this section, we generalize this estimate to the case where the base point is different from the origin.
The corresponding results under the stronger assumption are obtained in \cite[Section 7]{Mat2} and
the arguments here are carried out similarly by adding necessary modifications.
We prepare notation used hereafter.

\begin{definition}
For $\mu \in \Bel(\D)$, the 
quasiconformal self-homeomorphism of $\D$ whose complex dilatation coincides with $\mu$
having the normalization such that its quasisymmetric extension to $\mathbb S$ fixes 
three distinct points $1, i, -1$ is denoted by $f^\mu$.
For $\nu \in \Bel(\D)$,
the complex dilatation of the composition $f^\mu \circ f^\nu$ is denoted by $\mu \ast \nu$, and
that of the inverse $(f^\nu)^{-1}$ is denoted by $\nu^{-1}$. 
\end{definition}

The following inequality is a fundamental tool, which is due to Yanagishita \cite[Lemma 3.1, Proposition 3.2]{Yan}
obtained by applying the argument of Astala and Zinsmeister \cite{AZ}.
We note that $\rho_\Omega$ denotes the hyperbolic density on 
a domain $\Omega$ in $\widehat{\mathbb C}$.

\begin{proposition}\label{yanagishita}
For Beltrami coefficients $\mu$ and $\nu$ in $\Bel(\D)$, let $f_\mu$ and $f_\nu$ be the quasiconformal homeomorphisms of $\Chat$
that are conformal on $\D^*$ with the normalization
$f_\mu(\infty)=f_\nu(\infty)=\infty$ and $\lim_{z \to \infty} f_\mu'(z)=\lim_{z \to \infty} f_\nu'(z)=1$
and have the complex dilatations $\mu$ and $\nu$ respectively on $\D$.
Let $\Omega=f_\nu(\D)$ and $\Omega^*=f_\nu(\D^*)$. Then,
\begin{align*}
|S_{f_\mu \circ f_\nu^{-1}|_{\Omega^*}}(\zeta)| 
&\leq \frac{3\rho_{\Omega^*}(\zeta)}{\sqrt{\pi}}
\left(\int_\Omega \frac{|\mu(f_\nu^{-1}(w))-\nu(f_\nu^{-1}(w))|^2}{(1-|\mu(f_\nu^{-1}(w))|^2)(1-|\nu(f_\nu^{-1}(w))|^2)}
\frac{dudv}{|w-\zeta|^4}\right)^{1/2}
\end{align*}
holds for $\zeta \in \Omega^*$.
\end{proposition}

By modifying the norm $\Vert \cdot \Vert_{\infty,\alpha}$ of the Bers embedding to the norm 
$\Vert \cdot \Vert_{\infty,\alpha'}$ with a smaller exponent $\alpha' \in (0,\alpha]$,
we have the following basic fact for estimating the difference. 

\begin{lemma}\label{modification}
For $\mu, \nu \in \Bel(\D)$, if $\Vert \mu-\nu \Vert_{\infty,\alpha}<\infty$,
then there exists some $\alpha' \in (0,\alpha]$ such that
$\Vert \Phi(\mu)-\Phi(\nu) \Vert_{\infty,\alpha'} <\infty$.
\end{lemma}

\begin{proof}
The limiting case of the Mori theorem (\cite[Section III.C]{Ah0}) implies that 
$$
\frac{1}{A}(1-|z|)^K \leq 1-|f^\nu(z)| \leq A (1-|z|)^{1/K} 
$$
for some constant $A \geq 1$, where $K=(1+k)/(1-k)$ with $k=\Vert \nu \Vert_\infty$.
Then, we have that 
$$
\rho_{\D}^{-\alpha}(z) \leq a \rho_{\D}^{-\alpha/K}(f^\nu(z))
$$
for some constant $a>0$.
Let
$f=f_\nu \circ (f^\nu)^{-1}$, where $f_\nu$ is as in Proposition \ref{yanagishita}.
This map $f$ is a conformal homeomorphism of $\D$ onto $\Omega$, 
but extending $f^\nu$ to $\D^*$ by the reflection
with respect to $\S1$, we may assume that $f$
extends to a quasiconformal homeomorphism of $\Chat$ with $f(\D^*)=\Omega^*$. 

It is known that there is a constant $B \geq 1$ such that
$$
\frac{1}{B}(1-|z|)^{k} \leq |f'(z)| \leq B(1-|z|)^{-k}
$$
for $z \in \D$ (see Pommerenke \cite[p.125]{P}).
It follows that
there is a constant $b>0$ such that
$$
\rho_{\D}^{-\alpha/K}(f^\nu(z)) \leq b \rho_{\Omega}^{-\alpha /(K(1+k))}(f_\nu(z)).
$$
By the definition of the norm, we have
$$
|\mu(z)-\nu(z)| \leq \rho_{\D}^{-\alpha}(z) \Vert \mu-\nu \Vert_{\infty,\alpha} 
$$
for $z \in \D$.
For $w=f_\nu(z) \in \Omega$, these inequalities yield
$$
|\mu(f_\nu^{-1}(w))-\nu(f_\nu^{-1}(w))| \leq ab \rho_{\Omega}^{-\alpha/(K(1+k))}(w) \Vert \mu-\nu \Vert_{\infty,\alpha}.
$$
Let $\widetilde \alpha=\alpha /(K(1+k))$.

By substituting this inequality to the integral in Proposition \ref{yanagishita}, we will estimate 
$$
\left(\int_\Omega \frac{\rho_{\Omega}^{-2\widetilde \alpha}(w)}{|w-\zeta|^4}dudv\right)^{1/2}.
$$
Let $\eta_\Omega(w)$ be the Euclidean distance from $w \in \Omega$ to $\partial \Omega$ and
$\eta_{\Omega^*}(\zeta)$ the Euclidean distance from $\zeta \in \Omega^*$ to $\partial \Omega^*$.
As a consequence from the Koebe one-quarter theorem, we see that both $\rho_{\Omega}(w)\eta_{\Omega}(w)$
and $\rho_{\Omega^*}(\zeta)\eta_{\Omega^*}(\zeta)$ are bounded below by $1/2$.
We have
$$
\rho_{\Omega}^{-2\widetilde \alpha}(w) \leq 4\eta_{\Omega}^{2\widetilde \alpha}(w) \leq 4|w-\zeta|^{2\widetilde \alpha}
$$
for every $w \in \Omega$ and for every $\zeta \in \Omega^*$.
Hence, the integral can be estimated as
\begin{align*}
\int_\Omega \frac{\rho_{\Omega}^{-2\widetilde \alpha}(w)}{|w-\zeta|^4}dudv &\leq  
4\int_\Omega \frac{dudv}{|w-\zeta|^{4-2\widetilde \alpha}}\\
&\leq
4\int_{|w-\zeta| \geq \eta_{\Omega^*}(\zeta)}
\frac{dudv}{|w-\zeta|^{4-2\widetilde \alpha}}\\
&=
\frac{8\pi}{2-2\widetilde \alpha} \cdot
\frac{1}{\eta_{\Omega^*}(\zeta)^{2-2\widetilde \alpha}}
\leq
\frac{16 \pi}{1-\alpha}\cdot \rho_{\Omega^*}(\zeta)^{2-2\widetilde \alpha}.
\end{align*}

Substituting this estimate for the inequality of Proposition \ref{yanagishita}, we have
$$
\rho^{-2}_{\Omega^*}(\zeta)|S_{f_\mu \circ f_\nu^{-1}|_{\Omega^*}}(\zeta)| \leq 
\frac{12ab\Vert \mu-\nu \Vert_{\infty,\alpha}}{\sqrt{(1-\alpha)(1-\Vert \mu \Vert_\infty^2)(1-\Vert \nu \Vert_\infty^2)}}
\rho^{-\alpha /(K(1+k))}_{\Omega^*}(\zeta).
$$
For $\zeta=f_\nu(z)$ with $z \in \D^*$, the left side term is equal to
$$
\rho^{-2}_{\D^*}(z)|S_{f_\mu|_{\D^*}}(z)-S_{f_\nu|_{\D^*}}(z)|.
$$
For the right side term, we consider
the quasiconformal homeomorphism $f_\nu$ of $\Chat$ that is conformal on $\D^*$. This satisfies
$$
(1-|z|^{-2})^{k} \leq |f_\nu'(z)| \leq (1-|z|^{-2})^{-k}
$$
for $z \in \D^*$ (see Becker \cite[p.60]{Bec}). 
Then, there is a constant $b'>0$ such that
$$
\rho^{-\alpha /(K(1+k))}_{\Omega^*}(f_\nu(z)) \leq b' \rho^{-\alpha /K^2}_{\D^*}(z). 
$$
Therefore, the above inequality turns out to be
$$
\rho^{-2+(\alpha /K^2)}_{\D^*}(z)|S_{f_\mu|_{\D^*}}(z)-S_{f_\nu|_{\D^*}}(z)| \leq 
\frac{12abb'\Vert \mu-\nu \Vert_{\infty,\alpha}}{\sqrt{(1-\alpha)(1-\Vert \mu \Vert_\infty^2)(1-\Vert \nu \Vert_\infty^2)}}.
$$
This shows that $\Vert \mu-\nu \Vert_{\infty,\alpha}<\infty$ implies 
$\Vert \Phi(\mu)-\Phi(\nu) \Vert_{\infty,\alpha'}<\infty$ for $\alpha'=\alpha/K^2$.
\end{proof}

Under the stronger assumption that $\nu \in \Bel^{\varepsilon}_0(\D)$ for some $\varepsilon \in (0,1)$
in Lemma \ref{modification},
we have the following stronger consequence than the above. 
In particular, the stronger estimate for the hyperbolic density is also obtained in this argument. See 
\cite[Lemma 7.3]{Mat2}.

\begin{proposition}\label{epsilon-free}
$(1)$ For $\nu \in \Bel_0^{\varepsilon}(\D)$, there exists a constant $A \geq 1$  
such that
$$
A^{-1} \rho_{\D}(z) \leq \rho_{\D}(f^\nu(z)) \leq A \rho_{\D}(z) \quad (z \in \D).
$$ 
$(2)$ For $\mu \in \Bel(\D)$ and $\nu \in \Bel^{\varepsilon}_0(\D)$, if $\Vert \mu-\nu \Vert_{\infty,\alpha}<\infty$,
then
$$
\Vert \Phi(\mu) -\Phi(\nu)\Vert_{\infty, \alpha}<\infty.
$$
\end{proposition}

Next, we consider the difference of the norms of Beltrami coefficients under the right translation defined below.
The exponent $\alpha$ should be changed to some $\alpha'$ also in this case, but
this is not necessary when $\nu$ is promoted as above.

\begin{definition}
The right translation $r_{\nu}$ on $\Bel(\D)$
is defined by $r_{\nu}(\mu)=\mu \ast \nu^{-1}$ for any $\mu$ and $\nu$ in $\Bel(\D)$,
which satisfies $\pi \circ r_\nu=R_{\pi(\nu)} \circ \pi$,
where $R_\tau:T \to T$ is the base point change automorphism sending $\tau$ to $o$.
\end{definition}

\begin{proposition}\label{cancel}
For $\mu_1, \mu_2, \nu \in \Bel(\D)$, if $\Vert \mu_1-\mu_2 \Vert_{\infty,\alpha}<\infty$,
then there is some $\alpha' \in (0,\alpha]$ such that 
$$
\Vert r_\nu(\mu_1)-r_\nu(\mu_2) \Vert_{\infty,\alpha'} <\infty.
$$
Moreover, if $\nu \in \Bel_0^{\varepsilon}(\D)$ for some $\varepsilon \in (0,1)$ in addition, then
$$
\Vert r_\nu(\mu_1)-r_\nu(\mu_2) \Vert_{\infty,\alpha} <\infty.
$$
\end{proposition}

\begin{proof}
We have the following inequality for $\zeta=f^\nu(z)$:
\begin{align*}
|r_\nu(\mu_1)(\zeta)-r_\nu(\mu_2)(\zeta)|&=|\mu_1 \ast \nu^{-1}(\zeta)-\mu_2 \ast \nu^{-1}(\zeta)|\\
&= \left|\frac{\mu_1(z)-\nu(z)}{1-\overline{\nu(z)}\mu_1(z)}-\frac{\mu_2(z)-\nu(z)}{1-\overline{\nu(z)}\mu_2(z)}\right|\\
&= \frac{|\mu_1(z)-\mu_2(z)|(1-|\nu(z)|^2)}{|1-\overline{\nu(z)}\mu_1(z)||1-\overline{\nu(z)}\mu_2(z)|}
\leq \frac{|\mu_1(z)-\mu_2(z)|}{\sqrt{(1-|\mu_1(z)|^2)(1-|\mu_2(z)|^2)}}.
\end{align*}

For $\nu \in \Bel(\D)$, the Mori theorem as in the proof of
Lemma \ref{modification} implies that there are some $\alpha' \in (0,\alpha]$ and $c>0$ such that
$\rho_{\D}^{\alpha'}(\zeta) \leq c \rho_{\D}^{\alpha}(z)$.
Hence, we have the first statement.
For $\nu \in \Bel_0^{\varepsilon}(\D)$, the better estimate as in Proposition \ref{epsilon-free} (1)
shows that there is some $c'>0$ such that
$\rho_{\D}^{\alpha}(\zeta) \leq c' \rho_{\D}^{\alpha}(z)$. This yields the second statement.
\end{proof}

Finally, the combination of Lemma \ref{modification} and Proposition \ref{cancel} yields 
the following.

\begin{theorem}\label{base}
For $\mu_1, \mu_2, \nu \in \Bel(\D)$, if $\Vert \mu_1-\mu_2 \Vert_{\infty,\alpha}<\infty$,
then there is some $\alpha' \in (0,\alpha]$ such that
$$
\Vert \Phi(r_\nu(\mu_1))-\Phi(r_\nu(\mu_2)) \Vert_{\infty,\alpha'} <\infty.
$$
Moreover, if $\mu_2, \nu \in \Bel_0^{\varepsilon}(\D)$ for some $\varepsilon \in (0,1)$
in addition, then
$$
\Vert \Phi(r_\nu(\mu_1))-\Phi(r_\nu(\mu_2)) \Vert_{\infty,\alpha} <\infty.
$$
\end{theorem}

\begin{proof}
By the first statement of Proposition \ref{cancel},
we have $\Vert r_\nu(\mu_1)-r_\nu(\mu_2) \Vert_{\infty,\widetilde \alpha}<\infty $ for some $\widetilde \alpha \in (0,\alpha]$.
Then, Lemma \ref{modification} gives
$$
\Vert \Phi(r_\nu(\mu_1))-\Phi(r_\nu(\mu_2)) \Vert_{\infty,\alpha'} <\infty
$$
for some $\alpha' \in (0,\widetilde \alpha]$.
If $\mu_2, \nu \in \Bel_0^{\varepsilon}(\D)$, then $r_\nu(\mu_2) \in \Bel_0^{\varepsilon}(\D)$
because $\Bel_0^{\varepsilon}(\D)$ constitutes a group under the operation $\ast$.
This is due to the formula 
$$
|r_\nu(\mu_2)(\zeta)|=|\mu_2 \ast \nu^{-1}(\zeta)|=\left| \frac{\mu_2(z)-\nu(z)}{1-\overline{\nu(z)}\mu_2(z)}\right |
\quad (\zeta=f^\nu(z))
$$ 
and Proposition \ref{epsilon-free} (1). 
We apply the second statement of Proposition \ref{cancel} to obtain
$$
\Vert r_\nu(\mu_1)-r_\nu(\mu_2) \Vert_{\infty,\alpha} <\infty.
$$
Then, we have that
$$
\Vert \Phi(r_\nu(\mu_1))-\Phi(r_\nu(\mu_2)) \Vert_{\infty,\alpha} 
<\infty
$$
by Proposition \ref{epsilon-free} (2).
\end{proof}

\section {The rigidity theorem}\label{4}

In this section, we prove our rigidity theorem by extending
the arguments in Theorem \ref{th1} to Teich\-m\"ul\-ler spaces of circle diffeomorphisms.
For the proof, the Bers embedding of the universal Teich\-m\"ul\-ler space 
$\beta:T \to B(\D^*)$ also
plays a significant role.
By this embedding,
the action of a M\"obius group $\Gamma$ on $T$ is realized as a linear isometric action on the Banach space $B(\D^*)$.

\begin{theorem}[rigidity]\label{th2}
Let $\Gamma$ be a subgroup of $\Mob(\S1)$ that contains a hyperbolic element.
If $f \Gamma f^{-1} \subset \Diff^{1+\alpha}_+(\S1)$ for $f \in \Sym$, then $f \in \Diff^{1+\alpha}_+(\S1)$.
\end{theorem}

\begin{corollary}\label{cor-rigidity}
If a non-abelian infinite subgroup $G \subset \Diff^{1+\alpha}_+(\S1)$ fixes $[f] \subset T_0$,
then $[f] \in T_0^\alpha$.
\end{corollary}

The strategy of the proof is twofold: Self-improvement arguments and the proof for a hyperbolic cyclic subgroup.
The former one is based on the following claim, which is a reformulation of Theorem \ref{base}.

\begin{proposition}\label{sub}
Let $g$ belong to $\Diff^{1+\alpha}_+(\S1)$.
For every $f \in \QS$, there is some $\alpha' \in (0,\alpha]$ such that
$$
\beta([g \circ f]) \in {\beta([f])}+B^{\alpha'}_0(\D^*).
$$
Moreover, if $f \in \Diff^{1+\varepsilon}_+(\S1)$ for some $\varepsilon \in (0,1)$ in addition, then
$$
\beta([g \circ f]) \in {\beta([f])}+B^{\alpha}_0(\D^*).
$$
\end{proposition}

\begin{proof}
By Theorem \ref{main}, we choose $\mu \in \Bel^\alpha_0(\D)$ such that $\pi(\mu)=[g]$.
For the first case, we also choose $\nu \in \Bel(\D)$ such that $\pi(\nu)=[f^{-1}]$.
Then, $\beta([g \circ f])=\Phi(r_\nu(\mu))$ and $\beta([f])=\Phi(r_\nu(0))$.
The first statement of Theorem \ref{base} implies that
$$
\Vert \beta([g \circ f])-\beta([f]) \Vert_{\infty,\alpha'} 
<\infty
$$
for some $\alpha' \in (0,\alpha]$.
For the second case, we choose $\nu \in \Bel_0^{\varepsilon}(\D)$ such that $\pi(\nu)=[f^{-1}]$.
The second statement of Theorem \ref{base} implies that
$$
\Vert \beta([g \circ f])-\beta([f]) \Vert_{\infty,\alpha} 
<\infty.
$$
Thus, we obtain the consequences.
\end{proof}

Under the circumstances of Theorem \ref{th2},
let $\varphi=\beta([f])$, which belongs to $B_0(\D^*)$ by Proposition \ref{gs}.
The assumption $f\Gamma f^{-1} \subset \Diff^{1+\alpha}_+(\S1)$ is equivalent to
the existence of $g \in \Diff^{1+\alpha}_+(\S1)$ with $f \circ \gamma=g \circ f$
for every $\gamma \in \Gamma$. Then, the first statement of
Proposition \ref{sub} yields that
$\gamma^*\varphi-\varphi \in B^{\alpha'}_0(\D^*)$.
For the present, we want to obtain $\varphi \in B^{\alpha'}_0(\D^*)$ $(f \in \Diff^{1+\alpha'}_+(\S1))$
for some $\alpha' \in (0,\alpha]$
from this condition. Then, this will be improved to $\varphi \in B^{\alpha}_0(\D^*)$ by the second statement of Proposition \ref{sub},
which we call self-improvement arguments.

Our second strategy is to consider just one hyperbolic element $\gamma \in \Gamma$ 
to show that $\varphi \in B^{\alpha'}_0(\D^*)$.
We note that if $\gamma$ is parabolic, Lemma \ref{key} below does not work.
Choose any hyperbolic element $\gamma \in \Gamma$  
and set 
$$
\psi=\gamma^* \varphi-\varphi \in B^{\alpha'}_0(\D^*).
$$ 
Then, by similar arguments for proving Theorem \ref{th1}, we have the following
representation of $\varphi$.

\begin{proposition}\label{abel}
$\varphi(z)=-\sum_{i=0}^{\infty} (\gamma^*)^i \psi(z)=\sum_{i=1}^{\infty} (\gamma^*)^{-i}\psi(z)$ for each $z \in \D^*$.
\end{proposition}

\begin{proof}
For each $i \in \mathbb Z$, it holds $(\gamma^*)^i \psi=(\gamma^*)^{i+1} \varphi-(\gamma^*)^{i}\varphi$.
Summing up this from $i=0$ to $n \geq 0$, we have
$$
\sum_{i=0}^{n} (\gamma^*)^i \psi=(\gamma^*)^{n+1} \varphi-\varphi.
$$
Here, $\lim_{n \to +\infty}(\gamma^*)^{n+1} \varphi(z)=0$. Indeed, for each $z \in \D^*$,
$$
\rho_{\D^*}^{-2}(z)|(\gamma^*)^{n+1}\varphi(z)|
=\rho_{\D^*}^{-2}(\gamma^{n+1}(z))|\varphi(\gamma^{n+1}(z))|,
$$
and the right side term converges to $0$ as $n \to \infty$ because $\varphi \in B_0(\D^*)$.
Thus, $\varphi(z)=-\sum_{i=0}^{\infty} (\gamma^*)^i \psi(z)$ follows.
If we sum up the above equation from $i=-1$ to $-n \leq -1$ and take the limit as $n \to \infty$, 
then we can obtain the second equation in
the same reason.
\end{proof}

Using this representation of $\varphi$ in terms of $\psi \in B^{\alpha'}_0(\D^*)$,
we see that $\varphi=\beta([f])$ also belongs to $B^{\alpha'}_0(\D^*)$ as follows.

\begin{lemma}\label{key}
If $\varphi(z)=-\sum_{i=0}^{\infty} (\gamma^*)^i \psi(z)=\sum_{i=1}^{\infty} (\gamma^*)^{-i}\psi(z)$
for a hyperbolic element $\gamma \in \Mob(\D^*)$ 
and if $\psi \in B^{\alpha}_0(\D^*)$, then $\varphi \in B^{\alpha}_0(\D^*)$.
\end{lemma}

\begin{proof}
We take a M\"obius transformation $h$ that maps $\D^*$ to the upper half-plane $\H2$, 
the attracting fixed point $a_\gamma$ of $\gamma$ to $0$ and 
the repelling fixed point $r_\gamma$ of $\gamma$ to $\infty$. 
We use the Banach space of holomorphic quadratic differentials
$$
B^\alpha_0(\H2)=\{\phi \in \Hol_2(\H2) \mid \Vert \phi \Vert_{\infty,\alpha}=\sup_{\zeta \in \H2} \rho_{\H2}^{-2+\alpha}(\zeta)|\phi(\zeta)|<\infty\},
$$
where $\rho_{\H2}(\zeta)=1/{\rm Im} \zeta$ is the hyperbolic density on $\H2$.
We set $\widetilde \psi=h_* \psi$ for $\psi \in B^\alpha_0(\D^*)$,
where $h_* \psi(\zeta)=\psi(h^{-1}(\zeta))(h^{-1})'(\zeta)^2$.
It satisfies 
$$
\rho_{\H2}^{-2}(\zeta)|\widetilde \psi(\zeta)|=\rho_{\D^*}^{-2}(z)|\psi(z)|
$$
for $\zeta=h(z)$. Moreover, there is a constant $C>0$ 
depending only on $h$ such that
$$
\rho_{\D^*}^{\alpha}(z)=\rho_{\H2}^{\alpha}(\zeta)|h'(z)|^{\alpha} \geq \frac{1}{C^{\alpha}}\rho_{\H2}^{\alpha}(\zeta)
$$
for every $z \in \D^*$ except in some neighborhood of $\infty$. This shows that 
$\widetilde \psi= h_* \psi \in B^{\alpha}_0(\H2)$ for every $\psi \in B^{\alpha}_0(\D^*)$. Namely,
$$
h_*:B^{\alpha}_0(\D^*) \to B^{\alpha}_0(\H2)
$$
is a bounded linear injection, but not surjective nor isometric; the norm $\Vert \cdot \Vert_{\infty, \alpha}$ 
on $B^{\alpha}_0(\cdot)$
is not M\"obius invariant. We will prove that $\widetilde \varphi= h_* \varphi \in B^{\alpha}_0(\H2)$.

We define $\widetilde \gamma=h_*\gamma \in \Mob(\H2)$ by the conjugate $h_*\gamma=h\gamma h^{-1}$.
We note that the attracting fixed point of $\widetilde \gamma$ is $0$ and 
the repelling fixed point is $\infty$; this can be represented as $\widetilde \gamma(\zeta)=\lambda\zeta$ for the multiplier 
$\lambda \in (0,1)$ of $\gamma$.
Then, $\widetilde \gamma^* \widetilde \psi=h_*(\gamma^* \psi)$, where $\widetilde \gamma^*\widetilde \psi$ means the pull-back 
of $\widetilde \psi$ by $\widetilde \gamma$ as a holomorphic quadratic differential. 

From the assumption $\varphi(z)=-\sum_{i=0}^{\infty} (\gamma^*)^i \psi(z)$,
it follows that
$\widetilde \varphi(\zeta)=-\sum_{i=0}^\infty (\widetilde \gamma^*)^i \widetilde \psi(\zeta)$.
Here
$\Vert \widetilde \psi \Vert_{\infty,\alpha}<\infty$ and $\widetilde \gamma(\zeta)=\lambda \zeta$ $(0<\lambda <1)$.
Hence, we have
\begin{align*}
\rho^{-2}_{\mathbb H}(\zeta)|\widetilde \varphi(\zeta)|&=
\rho^{-2}_{\mathbb H}(\zeta) |\sum_{i=0}^\infty (\widetilde \gamma^*)^i \widetilde \psi(\zeta)|\\
&\leq \sum_{i=0}^\infty \rho^{-2}_{\mathbb H}(\zeta) |(\widetilde \gamma^*)^i \widetilde \psi(\zeta)|
= \sum_{i=0}^\infty \rho^{-2}_{\mathbb H}(\widetilde \gamma^i(\zeta)) |\widetilde \psi(\widetilde\gamma^i(\zeta))|\\
&\leq \Vert \widetilde \psi \Vert_{\infty,\alpha} \sum_{i=0}^\infty \rho^{-\alpha}_{\mathbb H}(\widetilde \gamma^i(\zeta))
= \Vert \widetilde \psi \Vert_{\infty,\alpha} \sum_{i=0}^\infty ({\rm Im} (\lambda^i \zeta))^\alpha\\
&= \Vert \widetilde \psi \Vert_{\infty,\alpha}\, ({\rm Im} \zeta)^\alpha \sum_{i=0}^\infty (\lambda^\alpha)^i
=\frac{\Vert \widetilde \psi \Vert_{\infty,\alpha}}{1-\lambda^\alpha} \,\rho^{-\alpha}_{\mathbb H}(\zeta).
\end{align*}
This gives $\widetilde \varphi \in B^{\alpha}_0(\mathbb H)$
(though this does not yet imply that $\varphi \in B^{\alpha}_0(\D^*)$).

To see $\varphi \in B^{\alpha}_0(\D^*)$, we use the other expression
$\varphi(z)=\sum_{i=1}^{\infty} (\gamma^*)^{-i}\psi(z)$.
We take a M\"obius transformation $e \in \Mob(\H2)$ with $e(\zeta)=-1/\zeta$.
Then, we have
$$
(e \circ h) \gamma^{-1} (e \circ h)^{-1}=\widetilde \gamma \quad {\rm or} \quad (e \circ h)_*\gamma^{-1}=\widetilde \gamma.
$$
Let $\widetilde \psi_1=(e \circ h)_*\psi$ and $\widetilde \varphi_1=(e \circ h)_*\varphi$.
As before $\widetilde \psi_1$ belongs to $B^\alpha_0(\H2)$, and from 
$\varphi(z)=\sum_{i=1}^{\infty} (\gamma^*)^{-i}\psi(z)$, it follows that
$\widetilde \varphi_1(\zeta)=\sum_{i=1}^{\infty} (\widetilde \gamma^*)^{i}\widetilde \psi_1(\zeta)$.
Then,
$$
\rho^{-2}_{\mathbb H}(\zeta)|\widetilde \varphi_1(\zeta)| \leq
\frac{\lambda^\alpha \Vert \widetilde \psi_1 \Vert_{\infty,\alpha}}{1-\lambda^\alpha} \,\rho^{-\alpha}_{\mathbb H}(\zeta),
$$
which gives $\widetilde \varphi_1 \in B^{\alpha}_0(\mathbb H)$.

Finally, we will conclude $\varphi \in B^{\alpha}_0(\D^*)$ from both 
$\widetilde \varphi \in B^{\alpha}_0(\mathbb H)$ and $\widetilde \varphi_1 \in B^{\alpha}_0(\mathbb H)$.
There are a constant $c>0$ and a 
neighborhood $U \subset \mathbb C$ of 
$r_\gamma$ with $a_\gamma \notin \overline{U}$
such that $|h'(z)| \leq c$ for every $z \in \D^*-U$. 
Hence,
\begin{align*}
\rho_{\D^*}^{-2}(z)|\varphi(z)|=\rho_{\H2}^{-2}(\zeta)|\widetilde \varphi(\zeta)|
&\leq \Vert \widetilde \varphi \Vert_{\infty,\alpha} \rho_{\H2}^{-\alpha}(\zeta)\\
&= \Vert \widetilde \varphi \Vert_{\infty,\alpha} \rho_{\D^*}^{-\alpha}(z)|h'(z)|^\alpha
\leq c^\alpha \Vert \widetilde \varphi \Vert_{\infty,\alpha} \rho_{\D^*}^{-\alpha}(z)
\end{align*}
for every $z \in \D^*-U$. Similarly,
there are a constant $c_1>0$ and a neighborhood $U_1 \subset \mathbb C$ of 
$a_\gamma$ with $U \cap U_1=\emptyset$
such that $|(e \circ h)'(z)| \leq c_1$ for every $z \in \D^*-U_1$.
Hence,
\begin{align*}
\rho_{\D^*}^{-2}(z)|\varphi(z)|=\rho_{\H2}^{-2}(\zeta_1)|\widetilde \varphi_1(\zeta_1)|
&\leq \Vert \widetilde \varphi_1 \Vert_{\infty,\alpha} \rho_{\H2}^{-\alpha}(\zeta_1)\\
&= \Vert \widetilde \varphi_1 \Vert_{\infty,\alpha} \rho_{\D^*}^{-\alpha}(z)|(e \circ h)'(z)|^\alpha
\leq c_1^\alpha \Vert \widetilde \varphi_1 \Vert_{\infty,\alpha} \rho_{\D^*}^{-\alpha}(z)
\end{align*}
for every $z \in \D^*-U_1$. Therefore, we have
$$
\rho_{\D^*}^{-2+\alpha}(z)|\varphi(z)| \leq 
\max \{c^\alpha \Vert \widetilde \varphi \Vert_{\infty,\alpha},c_1^\alpha \Vert \widetilde \varphi_1 \Vert_{\infty,\alpha}\}<\infty
$$
for every $z \in \D^*$,
which proves $\varphi \in B^\alpha_0(\D^*)$.
\end{proof}

Now, we finish the proof of the rigidity theorem. 

\medskip
\noindent
{\it Proof of Theorem \ref{th2}.}
By Proposition \ref{abel} and Lemma \ref{key}, we have $\varphi \in B^{\alpha'}_0(\D^*)$.
This condition implies that $f \in \Diff^{1+\alpha'}_+(\S1)$ by Theorem \ref{main}.
Having this,
we repeat the same argument as above from the beginning. At first, Proposition \ref{sub} yields that
$\gamma^*\varphi-\varphi \in B^{\alpha}_0(\D^*)$ in this turn. Then, we see that $\varphi \in B^{\alpha}_0(\D^*)$
again by Proposition \ref{abel} and Lemma \ref{key}, which shows that
$f \in \Diff^{1+\alpha}_+(\S1)$ by Theorem \ref{main}.
\qed

\section{Conjugation of a group of circle diffeomorphisms to a M\"obius group}\label{5}

With the aid of the rigidity theorem, we solve the conjugation problem of
a group of circle diffeomorphisms. 
We utilize the following integrable class of Beltrami coefficients.

\begin{definition}
A Beltrami coefficient $\mu \in \Bel(\D)$ is {\it $p$-integrable} for $p \geq 1$ if
$$
\Vert \mu \Vert_p^p=\int_{\D} |\mu(z)|^p \rho_{\D}^2(z)dxdy <\infty.
$$
The space of all $p$-integrable Beltrami coefficients on $\D$ is denoted by $\Ael^p(\D)$.
\end{definition}

We can find an appropriate subspace of $T$ including $T^\alpha_0$ where $\Diff_+^{1+\alpha}(\mathbb S) \subset \QS$ acts.
The following Teichm\"uller spaces have been studied by Cui \cite{Cui}, Guo \cite{Guo}, Shen \cite{Sh},
Takhtajan and Teo \cite{TT},
Tang \cite{Tang} and Yanagishita \cite{Yan} among others.

\begin{definition}
A quasisymmetric homeomorphism $g:\S1 \to \S1$ belongs to $\Sym^p$ for $p \geq 1$ if
$g$ has a quasiconformal extension $\widetilde g: \D \to \D$ whose complex dilatation $\mu_{\widetilde g}$
belongs to $\Ael^p(\D)$. The {\it $p$-integrable Teichm\"uller space} $T^p$ is defined by
$$
T^p=\pi(\Ael^p(\D))=\Mob(\S1) \backslash \Sym^p \subset T.
$$
The topology on $T^p$ is induced by $\Vert \cdot \Vert_p+\Vert \cdot \Vert_\infty$ on $\Ael^p(\D)$.
\end{definition}

We also consider the space of all $p$-integrable holomorphic quadratic differentials on $\D^*$:
$$
A^p(\D^*)=\{\varphi \in \Hol_2(\D^*) \mid \Vert \varphi \Vert_p^p=\int_{\mathbb D^*} 
\rho^{2-2p}_{\D^*}(z)|\varphi(z)|^p dxdy <\infty \}.
$$
Concerning the inclusion relation between $A^p(\D^*)$ and $B(\D^*)$,
the following results are known. See \cite[Lemma 1]{Cui} and \cite[Lemma 2]{Guo} for example.

\begin{proposition}\label{inclusion}
For every $\varphi \in A^p(\D^*)$, it holds that
$\Vert \varphi \Vert_\infty \leq c_p \Vert \varphi \Vert_p$ for every $p \geq 1$,
where $c_p^p=(2p-1)/(4\pi)$. In particular, $A^p(\D^*) \subset B(\D^*)$.
\end{proposition}

Moreover, the inclusion of $A^p(\D^*)$ between $B_0^\alpha(\D^*)$ and $B_0(\D^*)$ is also known (cf. \cite{Cui}, \cite{Guo}).

\begin{proposition}\label{inclusion2}
In general $A^p(\D^*) \subset B_0(\D^*)$ for $p \geq 1$. If $p\alpha>1$ then $B_0^\alpha(\D^*) \subset A^p(\D^*)$.
\end{proposition}

\begin{proof}
We use a fact that Laurent polynomials with a zero of at least fourth order
at $\infty$ are dense in $A^p(\D^*)$. This fact follows from, e.g., the arguments in \cite[Lemma 1]{Ber}.
For every $\varphi \in A^p(\D^*)$, we choose a sequence of
Laurent polynomials $\{\varphi_n\}$ that converges to $\varphi$. Then, by Proposition \ref{inclusion},
$$
\Vert \varphi-\varphi_n \Vert_\infty \leq c_p \Vert \varphi-\varphi_n \Vert_p \to 0 \quad(n \to \infty).
$$
As $\varphi_n \in B_0(\D^*)$ and $B_0(\D^*)$ is a closed subspace of $B(\D^*)$, we have $\varphi \in B_0(\D^*)$.

Every $\varphi \in B_0^{\alpha}(\D^*)$ satisfies 
$|\varphi(z)| \leq \rho^{2-\alpha}_{\D^*}(z) \Vert \varphi \Vert_{\infty,\alpha}$ by definition.
Then,
$$
\Vert \varphi \Vert^p_p=\int_{\D^*} \rho^{2-2p}_{\D^*}(z)|\varphi(z)|^p dxdy \leq 
\left(\int_{\D^*} \rho^{2-p\alpha}_{\D^*}(z) dxdy \right)
\Vert \varphi \Vert_{\infty,\alpha}^p.
$$
The last term is integrable if $2-p\alpha<1$, which implies that $\varphi \in A^p(\D^*)$ if $p\alpha>1$.
\end{proof}
\medskip

It was proved in \cite[Theorem 2]{Cui} and \cite[Theorem 2]{Guo} that the Bers embedding $\beta$ of $T^p$ 
is a homeomorphism onto the image and satisfies 
$$
\beta(T^p)=\beta(T) \cap A^p(\D^*)
$$
for $p \geq 2$. This in particular implies that $T^p \subset T_0$, and hence
$\Sym^p \subset \Sym$.
Similarly, $\beta:T_0^\alpha \to B_0^\alpha(\D^*)$ is a homeomorphism onto the image and satisfies
$$
\beta(T_0^\alpha)=\beta(T) \cap B_0^\alpha(\D^*),
$$
which were shown in \cite[Theorem 7.1]{Mat2}.
If $p\alpha>1$, then $T_0^\alpha \subset T^p$ and hence
$\Diff_+^{1+\alpha}(\mathbb S) \subset \Sym^p$.

For every point $\tau =[f] \in T$,
the base point change automorphism $R_\tau:T \to T$ is defined by $[g] \mapsto [g \circ f^{-1}]$ as before.
It is known that if $\tau \in T^p$ $(p \geq 2)$ then
$R_\tau$ preserves $T^p$ (see \cite[Theorem 4]{Cui}, \cite[Lemma 3.4]{TT}, \cite[Proposition 5.1]{Yan}).
This can be alternatively expressed as a condition $\pi(r_\nu(\mu)) \in T^p$ for any $\mu$ and $\nu$ in $\Ael^p(\D)$. 
The canonical coordinate of $T^p$ at each $\tau \in T^p$ as the Banach manifold is given by
$$
\beta_\tau=\beta \circ R_{\tau}:T^p \to \beta(T) \cap A^p(\D^*).
$$

\begin{definition}
The {\it $p$-Weil--Petersson metric} $d_{WP}^p$ on $T^p$ $(p \geq 2)$ is a Finsler metric induced by
the norm $\Vert \cdot \Vert_p$ on the tangent space $T_\tau(T^p)$ at each $\tau \in T^p$,
which is identified with $A^p(\D^*)$ by the canonical coordinate $\beta_\tau$.
The distance induced by this metric is also denoted by $d_{WP}^p(\cdot,\cdot)$.
\end{definition}

\begin{remark}
The above metric can be alternatively defined by using the operator norm of $\varphi \in A^p(\D^*)$
acting on $A^q(\D^*)$ with $1/p+1/q=1$. This is done in \cite[Definition 6.5]{Mat3}. However, the ratio of the two norms is bounded
from above and below (see Kra \cite[p.90]{Kra}), and hence there is not much difference.
\end{remark}

The topology defined by $d_{WP}^p$ coincides with the original topology on $T^p$.
Continuity of the $p$-Weil--Petersson metric $d_{WP}^p$ on $T^p$ will be shown in Theorem \ref{continuous} later.
For $p=2$, it is known that $d_{WP}^2$ is $C^\infty$-differentiable.
From the definition, we see that $d_{WP}^p$ is invariant under $R_\tau$ for every $\tau \in T^p$.
In particular, 
$$
d_{WP}^p(\pi \circ r_\nu(\mu_1),\pi \circ r_\nu(\mu_2))=d_{WP}^p(\pi (\mu_1),\pi (\mu_2))
$$
for any $\mu_1, \mu_2, \nu \in \Ael^p(\D)$.
We also see that $\Sym^p \subset \QS$ acts on $(T^p, d_{WP}^p)$ isometrically.

For $p=2$, Cui \cite[Theorems 5, 6]{Cui} proved that $(T^2, d_{WP}^2)$ is complete and contractible.
See also Theorem \ref{complete}.
Later, Takhtajan and Teo \cite[Theorem 7.14]{TT} proved that the sectional curvature of $(T^2, d_{WP}^2)$
at every point is negative. Then, $(T^2, d_{WP}^2)$ is a {\it Cartan-Hadamard manifold} of infinite
dimension and it satisfies that any two points can be joined by a unique shortest geodesic
(see Lang \cite[Chapter IX, Corollary 3.11]{Lang}). We take any geodesic triangle $A_1A_2A_3$ on 
$(T^2, d_{WP}^2)$ and consider its comparison triangle $A_1^*A_2^*A_3^*$ on the Euclidean plane which has the same edge lengths.
Then, the corresponding angles satisfy the inequality $\angle A_i \leq \angle A_i^*$ $(i=1,2,3)$
(\cite[Chapter IX, Theorem 4.8]{Lang}). From this condition, we can assert that the triangle $A_1A_2A_3$ possesses
the ${\rm CAT}(0)$ property, and hence 
$(T^2, d_{WP}^2)$ is a {\it ${\rm CAT}(0)$ space} as a geodesic metric space (see Bridson and Haefliger \cite[Chapter II.1]{BH} 
and Ballmann \cite[Chapter 1, Section 3]{Bal}).
A complete metric space $(X,d)$ can be characterized 
as a ${\rm CAT}(0)$ space by holding the following {\it CN-inequality} 
(\cite[Chapter 1, 5.1]{Bal}):
for any $x,y \in X$, the midpoint $m$ between $x$ and $y$ satisfies
$$
d(z,m)^2 \leq \frac{1}{2}(d(z,x)^2+d(z,y)^2)-\frac{1}{4}d(x,y)^2
$$
for every $z \in X$.

Now we state our result on the conjugation problem as follows.

\begin{theorem}[trivial conjugation]\label{conjugate}
Let $G$ be an infinite non-abelian subgroup of $\Diff_+^{1+\alpha}(\mathbb S)$ with $\alpha \in (1/2,1)$.
Then, the following conditions are equivalent:
\begin{enumerate}
\item
There exists some $f \in \Diff_+^{1+\alpha}(\mathbb S)$ such that
$f G f^{-1} \subset \Mob(\mathbb S)$;
\item
The orbit of $G$ is bounded in $T^2$
with respect to $d_{WP}^2$;
\item
There exist positive constants 
$\kappa_2 <\infty$ and $\kappa_\infty<1$
such that
$$
{\rm (a)}\ \inf_{\pi(\mu)=[g]} \Vert \mu \Vert_2 \leq \kappa_2; \quad 
{\rm (b)}\ \inf_{\pi(\mu)=[g]} \Vert \mu \Vert_\infty \leq \kappa_\infty
$$
for all $g \in G$.
\end{enumerate}
Moreover, if $(1)$ holds, then such an $f \in \Diff^{1+\alpha}_+(\S1)$ is unique up to the post-composition of a
M\"obius transformation.
\end{theorem}

\begin{remark}
We can replace the above infima of the norms of the complex dilatations
with the norm of the complex dilatation $\mu_{\widetilde g}$ of the
barycentric extension $\widetilde g \in \QC(\D)$ of $g$ introduced by Douady and Earle \cite{DE}.
For condition (b), we can find this fact in \cite[Proposition 7]{DE}, 
and for condition (a), we find in Cui \cite[Theorem 1]{Cui}.
A subgroup $G \subset \QS$ satisfying condition (b) is called {\it uniformly quasisymmetric}.
Our proof of Theorem \ref{conjugate} does not rely on the result of Markovic \cite{Mar}.
\end{remark}

For the proof of Theorem \ref{conjugate}, we use 
the following theorem, which can be obtained for $p \geq 2$ in general. 
This result has its own interest because it asserts 
a certain kind of metrically equivalent condition
between the $p$-Weil--Petersson distance $d_{WP}^p$
and the norm $\Vert \cdot \Vert_p+\Vert \cdot \Vert_\infty$ on ${\rm Ael}^p(\D)$ that defines the topology on $T^p$.
The inverse inequality will be also given later in Proposition \ref{p-norm}.

\begin{theorem}\label{WPdistance}
For every $\mu \in \Ael^p(\D)$, the $p$-Weil--Petersson distance satisfies
$$
d_{WP}^p(\pi(0),\pi(\mu)) \leq C \Vert \mu \Vert_p,
$$ 
where $C>0$ is a constant depending only on $\Vert \mu \Vert_\infty$.
\end{theorem}

\begin{remark}
Precisely speaking, the dependence of the constant $C$ should be stated as follows:
$C$ depends on $k$ for which $\Vert \mu \Vert_\infty \leq k$ is satisfied. 
Hereafter, the dependence of a constant on
some norm or distance is mentioned always in this sense.
\end{remark}

The proof of this theorem will be given in Section \ref{7} by dividing the arguments into several steps.
Assuming Theorem \ref{WPdistance} for the moment, we can prove Theorem \ref{conjugate}.

\medskip
\noindent
{\it Proof of Theorem \ref{conjugate}.}
$(3) \Rightarrow (2):$ By Theorem \ref{WPdistance}, uniform boundedness conditions (a) and (b)
imply that $G$ has a bounded orbit in $T^2$.

$(2) \Rightarrow (1):$
By the property of ${\rm CAT}(0)$ space,
if the orbit of $G$ in $T^2$ is bounded with respect to $d_{WP}^2$,
then $G$ has a fixed point $[f]$ in $T^2$. Indeed, any bounded subset has a unique circumcenter
by the CN-inequality for complete ${\rm CAT}(0)$ spaces (see Ballmann \cite[Chapter 1, 5.10]{Bal}) and 
the unique circumcenter of the orbit of an isometric action is clearly a fixed point. 
As $[f] \in T^2 \subset T_0$, Corollary \ref{cor-rigidity} shows that $[f] \in T_0^\alpha$.
Thus, we find $f \in \Diff^{1+\alpha}_+(\S1)$ with $f G f^{-1} \subset \Mob(\S1)$.

Suppose that $(1)$ holds. We will prove the uniqueness of $f$. 
Assume that there is another $f_1 \in \Diff^{1+\alpha}_+(\S1)$ such that
$f_1 G f_1^{-1}=:\Gamma_1 \subset \Mob(\S1)$. 
Then, $f \circ f_1^{-1} \in \Diff^{1+\alpha}_+(\S1) \subset \Sym$ conjugates $\Gamma_1$ into $\Mob(\S1)$.
From this,
Theorem \ref{th1} shows that $f \circ f_1^{-1} \in \Mob(\S1)$,
which yields the assertion.

$(1) \Rightarrow (3):$ Estimates of the norms $\Vert \cdot \Vert_2$ and $\Vert \cdot \Vert_\infty$ under composition
show that if $f G f^{-1} \subset \Mob(\mathbb S)$ for $f \in \Diff_+^{1+\alpha}(\mathbb S) \subset \Sym^2$, 
then $G$ satisfies conditions (a) and (b). Indeed, the inequality for $\Vert \cdot \Vert_\infty$ is well-known.
The inequality for $\Vert \cdot \Vert_2$ can be found in \cite[Lemma 2.7]{TT} and \cite[Proposition 5.1]{Yan},
but we will prove this below for the sake of convenience.

Let $\nu \in \Bel_0^\alpha(\D) \subset \Ael^2(\D)$ be the complex dilatation of the 
barycentric extension $\widetilde {f^{-1}}$ of $f^{-1} \in \Diff_+^{1+\alpha}(\mathbb S)$. 
For each $g \in G$, 
we set $\gamma=fgf^{-1} \in \Mob(\S1)$ and regard it also as an element of $\Mob(\D)$. 
Then, $\widetilde {f^{-1}} \gamma (\widetilde {f^{-1}})^{-1}$ is a quasiconformal extension of $g \in G$.
We will show that the $2$-norm of the complex dilatation of this quasiconformal homeomorphism is
uniformly bounded. The complex dilatation of $\widetilde {f^{-1}} \gamma$ can be denoted by $\gamma^*\nu$.
Then, the complex dilatation of $\widetilde {f^{-1}} \gamma (\widetilde {f^{-1}})^{-1}$ is $r_\nu(\gamma^*\nu)$.

We apply the formula for $r_\nu(\gamma^*\nu)$ as before.
Then, we have
$$
|r_\nu(\gamma^*\nu)(\zeta)| \leq \left|\frac{\gamma^*\nu(z)-\nu(z)}{1-\overline{\nu(z)}\gamma^*\nu(z)}\right|
\leq \frac{|\gamma^*\nu(z)-\nu(z)|}{1-\Vert \nu \Vert_\infty^2}
$$
for $\zeta=\widetilde {f^{-1}}(z)$. Here, we note that the Jacobian $J_{\widetilde {f^{-1}}}$ of $\widetilde {f^{-1}}$ satisfies
$$
\rho_{\D}^2(\widetilde {f^{-1}}(z))J_{\widetilde {f^{-1}}}(z) \leq C \rho_{\D}^2(z)
$$
for some constant $C>0$ depending only on $\Vert \nu \Vert_\infty$. See \cite[Theorem 2]{DE}
and \cite[Proposition 2.1]{Yan}. Therefore,
\begin{align*}
\Vert r_\nu(\gamma^*\nu) \Vert_2&=\left(\int_{\D}|r_\nu(\gamma^*\nu)(\zeta)|^2\rho_{\D}^2(\zeta)d\xi d\eta\right)^{1/2}\\
&\leq \frac{C^{1/2}}{1-\Vert \nu \Vert_\infty^2} \left(\int_{\D}|\gamma^*\nu(z)-\nu(z)|^2 \rho_{\D}^2(z)dxdy \right)^{1/2}\\
&\leq \frac{C^{1/2}}{1-\Vert \nu \Vert_\infty^2}(\Vert \gamma^*\nu \Vert_2+\Vert \nu \Vert_2).
\end{align*}
As $\Vert \gamma^*\nu \Vert_2=\Vert \nu \Vert_2$ by $\rho_{\D}^2(\gamma(z))J_{\gamma}(z)=\rho_{\D}^2(z)$,
we see that $\Vert r_\nu(\gamma^*\nu) \Vert_2$ is uniformly bounded.
\qed
\medskip

Theorem \ref{conjugate} can be generalized to some extent for an arbitrary $\alpha \in (0,1)$ as follows.
The proof will be given in Section \ref{7}.
Here, we define
$$
k_p(g)=\inf_{\pi(\mu)=[g]} \left( \int_{\D} \left(\frac{|\mu(z)|^2}{1-|\mu(z)|^2}\right)^{p/2} \rho_{\D}^2(z)dxdy \right)^{1/p}
$$
for every $g \in \Sym^p$ with $p \geq 2$.

\begin{theorem}\label{conjugate2}
If an infinite non-abelian subgroup $G \subset \Diff_+^{1+\alpha}(\mathbb S)$ for $\alpha \in (0,1)$ 
satisfies $k_p(g) \leq \varepsilon_p$
for all $g \in G$ and for 
a sufficiently small constant $\varepsilon_p>0$ depending only on $p \geq 2$ with $p\alpha>1$, 
then there exists $f \in \Diff_+^{1+\alpha}(\mathbb S)$
such that $f G f^{-1} \subset \Mob(\mathbb S)$.
\end{theorem}

Finally in this section, we record another consequence from
Theorem \ref{th2}. In the conjugation problem of Navas \cite{Nav}, 
uniform integrability of the Liouville cocycle is assumed for the 
group $G$ of circle diffeomorphisms. 
Here, for an orientation-preserving absolutely continuous self-homeomorphism $g$ of $\S1=\R/2\pi \Z$,
$$
c(g)(x,y)=\frac{g'(x) g'(y)}{4 \sin^2((g(x)-g(y))/2)}
-\frac{1}{4 \sin^2((x-y)/2)}
$$
is called the Liouville cocycle. We consider its integrable norm 
$$
\Vert c(g) \Vert_1=\int_{\S1 \times \S1}|c(g)(x,y)| dxdy.
$$

\begin{proposition}\label{cocycle}
Let $\Gamma$ be a subgroup of $\Mob(\S1)$ that contains a hyperbolic element.
If $f \Gamma f^{-1} \subset \Diff^{1+\alpha}_+(\S1)$ for an orientation-preserving absolutely continuous
self-homeo\-morphism $f$ of $\S1$ with $\Vert c(f) \Vert_1 <\infty$, then
$f \in \Diff^{1+\alpha}_+(\S1)$.
Moreover, $f$ belongs to $\Sym^p$ for any $p>1$.
\end{proposition}

\begin{proof}
As $\Vert c(f) \Vert_1 <\infty$,
we see that $f$ belongs to $\Sym$ by \cite[Theorem 5.1]{Mat4}. 
Hence, Theorem \ref{th2} shows that $f \in \Diff_+^{1+\alpha}(\mathbb S)$. 
Then, $f$ also belongs to $\Sym^p$ for any $p>1$ by \cite[Corollary 5.4]{Mat4}.
\end{proof}

\section{The norm estimate of the Bers projection}\label{6}
We prepare several claims for the proofs of the theorems stated in the previous section.
The following norm estimate of the Bers projection $\Phi:\Ael^p(\D) \to A^p(\D^*)$ at the origin is
crucial, which was essentially given by Cui \cite[Theorem 2]{Cui} for $p=2$ and similarly to $p > 2$ by
Guo \cite[Theorem 2]{Guo}.

\begin{lemma}\label{cui}
The Bers projection $\Phi$ satisfies
$$
\Vert \Phi(\mu) \Vert_p \leq \frac{3}{2}\left(\int_{\D}\left(\frac{|\mu(z)|^2}{1-|\mu(z)|^2}\right)^{p/2}\rho_{\D}^2(z)dxdy \right)^{1/p}
\leq \frac{3}{2\sqrt{1-\Vert \mu \Vert_\infty^2}} \Vert \mu \Vert_p
$$
for every $\mu \in \Ael^p(\D)$.
\end{lemma}

\begin{proof}
We put $\nu=0$ for Proposition \ref{yanagishita}. Then,
$$
|S_{f_\mu|_{\D^*}}(\zeta)| 
\leq \frac{3 \rho_{\D^*}(\zeta)}{\sqrt{\pi}}
\left(\int_{\D} \frac{|\mu(z)|^2}{1-|\mu(z)|^2}
\frac{dxdy}{|z-\zeta|^4}\right)^{1/2}.
$$
Here, applying the H\"older inequality for the integral, we have
\begin{align*}
& \quad \int_{\D} \frac{|\mu(z)|^2}{1-|\mu(z)|^2} \frac{dxdy}{|z-\zeta|^4} \\
&\leq 
\left(\int_{\D} \left(\frac{|\mu(z)|^2}{1-|\mu(z)|^2}\right)^{p/2} \frac{dxdy}{|z-\zeta|^4}\right)^{2/p} 
\left(\int_{\D} \frac{dxdy}{|z-\zeta|^4}\right)^{1-2/p}.
\end{align*}
Moreover, it is known that
$$
\int_{\D} \frac{dxdy}{|z-\zeta|^4}=\frac{\pi}{4}\,\rho_{\D^*}^2(\zeta); \quad 
\int_{\D^*} \frac{d\xi d \eta}{|z-\zeta|^4}=\frac{\pi}{4}\,\rho_{\D}^2(z).
$$
This shows that
\begin{align*}
& \quad \int_{\D^*}\rho_{\D^*}^{2-2p}(\zeta)|S_{f_\mu|_{\D^*}}(\zeta)|^p d\xi d\eta\\
&\leq \left(\frac{3}{\sqrt{\pi}}\right)^p \int_{\D^*} \rho_{\D^*}^{2-p}(\zeta)
\left[\int_{\D} \left(\frac{|\mu(z)|^2}{1-|\mu(z)|^2}\right)^{p/2} \frac{dxdy}{|z-\zeta|^4}\right]
\left(\frac{\pi}{4}\,\rho_{\D^*}^2(\zeta)\right)^{p/2-1}d\xi d\eta\\
&=\frac{4\cdot 3^p}{2^p \pi} \int_{\D} \left(\frac{|\mu(z)|^2}{1-|\mu(z)|^2}\right)^{p/2} \left[\int_{\D^*}\frac{d\xi d\eta}{|z-\zeta|^4}\right] dxdy\\
&=\left(\frac{3}{2}\right)^p  \int_{\D} \left(\frac{|\mu(z)|^2}{1-|\mu(z)|^2}\right)^{p/2}\rho_{\D}^2(z)dxdy,
\end{align*}
which implies the required estimate.
\end{proof}

\begin{remark}
The derivative of $\Phi$ at $0$ in the direction of $\mu \in \Bel(\D)$ can be represented by
$$
d_0\Phi(\mu)(z)=-\frac{6}{\pi}\int_{\D} \frac{\mu(\zeta)}{(\zeta-z)^4} d\xi d\eta \quad (z \in \D^*).
$$
See for instance \cite[Section 3.4.5]{Nag}. This defines a bounded linear operator 
$d_0\Phi$ from the tangent space of $\Ael^p(\D)$ to $A^p(\D^*)$. From Lemma \ref{cui}, we see that
its operator norm satisfies $\Vert d_0\Phi \Vert \leq 3/2$. The derivative $d_0\Phi(\mu)$  
can be alternatively represented by using the Bergman kernel 
$$
K_{\D^*}(z,\zeta)=\frac{12}{\pi} \frac{1}{(\bar z \zeta-1)^4}
$$
on $\D^*$; the above formula turns out to be the Bergman projection
$$
d_0\Phi(\mu)(z)=-\frac{1}{2}\int_{\D^*} (\zeta \zeta^*)^{-2} \mu(\zeta^*)
\overline{K_{\D^*}(z,\zeta)} d\xi d\eta \quad (z \in \D^*,\ \zeta^*=1/\bar \zeta \in \D).
$$
Then, by Kra \cite[p.90]{Kra}, the estimate $\Vert d_0\Phi \Vert \leq 3/2$ was already known.
\end{remark}

A holomorphic local section of $\Phi$ at the origin $0 \in B(\D^*)$ can be given explicitly
by Ahlfors and Weill \cite{AW} (see also \cite[Theorem II.5.1]{Leh}). The following form is
the adaptation to the unit disk case.

\begin{theorem}\label{AWsection}
Let $U^\infty(1/2)$ be the open ball of the Banach space $B(\D^*)$ centered at the origin with radius $1/2$.
For every $\varphi \in U^\infty(1/2)$, let
$$
\sigma(\varphi)(z)=-2\rho_{\D^*}^{-2}(z^*)(zz^*)^2\varphi(z^*).
$$
Then, $\mu(z)=\sigma(\varphi)(z)$ belongs to $\Bel(\D)$ and satisfies $\Phi(\mu)=\varphi$.
Here, $z^*=1/\bar z \in \D^*$ is the reflection of $z \in \D$ with respect to $\S1$. Hence, $\sigma:U^\infty(1/2) \to B(\D^*)$
is a holomorphic local section of $\Phi$ around $0$.
\end{theorem}

A quasiconformal self-homeomorphism of $\D$ induced by the Ahlfors--Weill section is a diffeo\-morphism,
its complex dilatation is infinitely differentiable and it
has a convenient property for its Jacobian. This was discovered by
Takhtajan and Teo \cite[Lemma 2.5]{TT}, based on a fact that the partial derivative $\partial f^\mu(z)$ 
at each point $z \in \D$ converges to $1$ as $\mu \in \sigma(U^\infty(1/2))$ converges to $0$
within the Ahlfors--Weill section.
See Bers \cite[p.97]{B}.

\begin{proposition}\label{jacobian}
For every $\varepsilon >0$, there exists $\delta \in (0,1/2)$ such that
the quasiconformal homeomorphism $f^\mu$ of $\D$ with the complex dilatation $\mu$ or $\mu^{-1}$
in $\sigma(U^\infty(\delta))$ satisfies
$$
|\rho^2_{\D}(f^{\mu}(z))|\partial f^\mu(z)|^2-\rho^2_{\D}(z)| \leq \varepsilon \rho^2_{\D}(z)
$$
for every $z \in \D$. In particular,
there exists $\delta_0 \in (0,1/2)$ such that  
$f^\mu$ with $\mu\ {\rm or}\ \mu^{-1} \in \sigma(U^\infty(\delta_0))$ satisfies
$$
\rho^2_{\D}(f^{\mu}(z))J_{f^\mu}(z) \leq 2(1-|\mu(z)|^2)\rho^2_{\D}(z),
$$
where $J_{f^\mu}(z)=|\partial f^\mu(z)|^2-|\bar \partial f^\mu(z)|^2$ is the Jacobian of $f^\mu(z)$.
\end{proposition}

This result in particular implies that the Jacobian with respect to the hyperbolic metric is
estimated as  
$$
\rho_{\D}^2(w)dudv \leq 2 \rho_{\D}^2(z)dxdy \qquad (w=u+iv, z=x+iy)
$$
for $w=f^\mu(z)$ with $\mu\ {\rm or}\ \mu^{-1} \in \sigma(U^\infty(\delta_0))$. 
Hereafter, we choose $\delta_0$ as in the above lemma and fix it so that $\delta_0 \leq 1/4$.

The generalization of Lemma \ref{cui} can be also obtained. For $p=2$, this is essentially given by \cite[Lemma 2.9]{TT}.

\begin{lemma}\label{c-y}
Let $\mu \in \Bel(\D)$ be arbitrary and let $\nu \ {\rm or}\ \nu^{-1} \in \Bel(\D)$ be in $\sigma(U^\infty(\delta_0))$.
Then,
\begin{align*}
\Vert \Phi(\mu)-\Phi(\nu) \Vert_p &\leq 
12 \left(\int_{\D}\left(\frac{|\mu(z)-\nu(z)|^2}{(1-|\mu(z)|^2)(1-|\nu(z)|^2)}\right)^{p/2}\rho_{\D}^2(z)dxdy \right)^{1/p}\\
&\leq \frac{12}{\sqrt{(1-\Vert \mu \Vert_\infty^2)(1-\Vert \nu \Vert_\infty^2)}}\Vert \mu-\nu \Vert_p.
\end{align*}
\end{lemma}

\begin{proof}
Let $\Omega=f_\nu(\D)$ and $\Omega^*=f_\nu(\D^*)$. Applying the H\"older inequality
to the integral appearing Proposition \ref{yanagishita}, we have
\begin{align*}
& \quad \int_{\Omega} \frac{|\mu(f_\nu^{-1}(w))-\nu(f_\nu^{-1}(w))|^2}{(1-|\mu(f_\nu^{-1}(w))|^2)(1-|\nu(f_\nu^{-1}(w))|^2)} \frac{dudv}{|w-\zeta|^4} \\
&\leq 
\left(\int_{\Omega} \left(\frac{|\mu(f_\nu^{-1}(w))-\nu(f_\nu^{-1}(w))|^2}{(1-|\mu(f_\nu^{-1}(w))|^2)(1-|\nu(f_\nu^{-1}(w))|^2)}\right)^{p/2} \frac{dudv}{|w-\zeta|^4}\right)^{2/p} 
\left(\int_{\Omega} \frac{dudv}{|w-\zeta|^4}\right)^{1-2/p}.
\end{align*}
Here, we have the following inequalities in this case by the same arguments in the proof of Lemma \ref{modification}:
$$
\int_{\Omega} \frac{dudv}{|w-\zeta|^4} \leq 4\pi \rho_{\Omega^*}^2(\zeta); \quad 
\int_{\Omega^*} \frac{d\xi d\eta}{|w-\zeta|^4} \leq 4\pi \rho_{\Omega}^2(w).
$$
This shows that
\begin{align*}
&\quad\int_{\Omega^*}\rho_{\Omega^*}^{2-2p}(\zeta)|S_{f_\mu \circ f_\nu^{-1}|_{\Omega^*}}(\zeta)|^p d\xi d\eta\\
&\leq \left(\frac{3}{\sqrt{\pi}}\right)^p \int_{\Omega^*} \rho_{\Omega^*}^{2-p}(\zeta)
\left[\int_{\Omega} \left(\frac{|\mu(f_\nu^{-1}(w))-\nu(f_\nu^{-1}(w))|^2}{(1-|\mu(f_\nu^{-1}(w))|^2)(1-|\nu(f_\nu^{-1}(w))|^2)}\right)^{p/2} \frac{dudv}{|w-\zeta|^4}\right]\\
& \quad \times \left(4\pi\rho_{\Omega^*}^2(\zeta)\right)^{p/2-1}d\xi d\eta\\
&=\frac{6^p}{4\pi} \int_{\Omega} \left(\frac{|\mu(f_\nu^{-1}(w))-\nu(f_\nu^{-1}(w))|^2}{(1-|\mu(f_\nu^{-1}(w))|^2)(1-|\nu(f_\nu^{-1}(w))|^2)}\right)^{p/2} \left[\int_{\Omega^*}\frac{d\xi d\eta}{|w-\zeta|^4}\right] dudv\\
&\leq 6^p  \int_{\Omega} \left(\frac{|\mu(f_\nu^{-1}(w))-\nu(f_\nu^{-1}(w))|^2}{(1-|\mu(f_\nu^{-1}(w))|^2)(1-|\nu(f_\nu^{-1}(w))|^2)}\right)^{p/2}\rho_{\Omega}^2(w)dudv.
\end{align*}

By the change of variable $z=f_{\nu}^{-1}(\zeta)$ for $\zeta \in \Omega^*$ and the Cayley identity for Schwarzian derivatives, 
the first term in the above inequality equals to
$$
\int_{\D^*}\rho_{\D^*}^{2-2p}(z)|S_{f_\mu|_{\D^*}}(z)-S_{f_\nu|_{\D^*}}(z)|^p dxdy=\Vert \Phi(\mu)-\Phi(\nu) \Vert_p^p.
$$
Moreover, $f_\nu:\D \to \Omega$ can be given by
the composition of the quasiconformal self-homeomorphism $f^\nu$ of $\D$ and
a conformal homeomorphism $\D \to \Omega$. By Proposition \ref{jacobian}, the Jacobian of $f^\nu$ satisfies
$\rho^2_{\D}(f^{\mu}(z))J_{f^\mu}(z) \leq 2\rho^2_{\D}(z)$. Hence,  
the change of variable $z=f_{\nu}^{-1}(w)$ for $w \in \Omega$ is applied under
$$
\rho_{\Omega}^2(w)dudv \leq 2 \rho_{\D}^2(z)dxdy.
$$
Thus, the last term in the above inequality is bounded by
$$
2 \cdot 6^p \int_{\D} \left(\frac{|\mu(z)-\nu(z)|^2}{(1-|\mu(z)|^2)(1-|\nu(z)|^2)}\right)^{p/2}\rho_{\D}^2(z)dxdy.
$$
By taking the $p$-th root, we obtain the desired inequality.
\end{proof}

\begin{remark}
A similar argument to the above proof
can be found in Yanagishita \cite[Proposition 3.2]{Yan}. In his paper,
instead of taking $\nu$ from $\sigma(U^\infty(\delta_0))$, he assumes
that $\nu$ is obtained by the barycentric extension. Then, by the estimate of the Jacobian 
with respect to the hyperbolic metric as in the proof of Theorem \ref{conjugate},
we can show a similar result to Lemma \ref{c-y} also under this assumption on $\nu$.
This is used to prove the continuity of $\Phi$. 
See also Cui \cite{Cui} and Tang \cite{Tang}.
\end{remark}

\section{Proofs of theorems}\label{7}
In this section, we give the proofs of Theorems \ref{WPdistance} and \ref{conjugate2}.
First, as another consequence from Proposition \ref{jacobian}, we see that
the right translation $r_\nu$ of $\Bel(\D)$ restricted to $\Ael^p(\D)$ is locally Lipschitz continuous
if $\nu$ or $\nu^{-1}$ is given by the Ahlfors--Weill section of small norm.
Compare with Proposition \ref{cancel}.

\begin{proposition}\label{p-est}
We take $\mu_1, \mu_2$ in $\Bel(\D)$ and $\nu$ or $\nu^{-1}$ in $\sigma(U^\infty(\delta_0))$. 
Then, 
$$
\Vert r_\nu(\mu_1)-r_\nu(\mu_2) \Vert_p \leq C_0 \Vert \mu_1-\mu_2 \Vert_{p},
$$
where $C_0>0$ is a constant depending only on $\Vert \mu_1 \Vert_\infty$ and $\Vert \mu_2 \Vert_\infty$.
\end{proposition}

\begin{proof}
As before, we use the following inequality for $w=f^\nu(z)$:
$$
|r_\nu(\mu_1)(w)-r_\nu(\mu_2)(w)|
\leq \frac{|\mu_1(z)-\mu_2(z)|}{\sqrt{(1-\Vert \mu_1 \Vert_\infty^2)(1-\Vert \mu_2 \Vert_\infty^2)}}.
$$
Then, Proposition \ref{jacobian} is again applied to show that
\begin{align*}
&\quad \int_{\D} |r_\nu(\mu_1)(w)-r_\nu(\mu_2)(w)|^p \rho_{\D}^2(w) dudv\\
&\leq
\frac{2}{(1-\Vert \mu_1 \Vert_\infty^2)^{p/2}(1-\Vert \mu_2 \Vert_\infty^2)^{p/2}}
\int_{\D}|\mu_1(z)-\mu_2(z)|^p \rho_{\D}^2(z) dxdy.
\end{align*}
The assertion is now clear.
\end{proof}

\begin{remark}
More generally, the estimate in Proposition \ref{p-est} is still valid 
if $\nu \in \Bel(\D)$ is given by the composition of elements in 
$\sigma(U^\infty(\delta_0))$. In this case, the constant $C_0$ also depends on $\Vert \nu \Vert_\infty$.
Based on this fact and Lemma \ref{c-y}, Takhtajan and Teo \cite[Theorem 2.13]{TT} proved 
the inclusion
$$
\Phi \circ r_\nu(\Ael^p(\D)) \subset \Phi(\nu^{-1})+A^p(\D^*)
$$
in the case of $p=2$. This is also true for $p>2$ in general, and in fact,
the equality holds by taking the intersection of the right side term with $\beta(T)=\Phi(\Bel(\D))$ (\cite[Theorem 4.1]{Mat7}). 
As we have mentioned before,
Proposition \ref{p-est} can be also obtained if we replace the assumption on $\nu$ with the condition that
it is given by the barycentric extension. See Yanagishita \cite[Proposition 5.1]{Yan}.
This has been already used in the proof of Theorem \ref{conjugate} $(1) \Rightarrow (3)$.
\end{remark}

Under this preparation, we start the proofs of the theorems here.

\medskip
\noindent
{\it Proof of Theorem \ref{WPdistance}.}
Let $\delta_0>0$ be the constant as in Proposition \ref{jacobian}.
We take the number of division $n \in \N$ greater than
$$
\frac{3\Vert \mu \Vert_\infty}{2\delta_0(1-\Vert \mu \Vert_\infty^2)},
$$
and set $t_i=i/n$ $(i=0,1,\ldots, n)$. Then, for every $i \geq 1$,
$$
\Vert r_{t_{i-1} \mu}(t_i\mu) \Vert_\infty \leq 
\frac{\Vert t_i \mu-t_{i-1} \mu \Vert_\infty}{1-\Vert t_i \mu \Vert_\infty \Vert t_{i-1} \mu \Vert_\infty}
\leq \frac{\Vert \mu \Vert_\infty}{n(1-\Vert \mu \Vert_\infty^2)}<\frac{2\delta_0}{3}.
$$
For the Bers projection $\Phi:\Bel(\D) \to B(\D^*)$, we define 
$$
\varphi_i=\Phi(r_{t_{i-1} \mu}(t_i\mu)).
$$
We note that $\Vert \Phi(\mu) \Vert_\infty \leq 3 \Vert \mu \Vert_\infty/2$ (\cite[Theorem II.3.2]{Leh}).
Then, $\Vert \varphi_i \Vert_\infty <\delta_0$, and hence $\varphi_i$ belongs to $U^\infty(\delta_0) \cap A^p(\D^*)$.

For integers $i$ and $k$ with $1 \leq i \leq n$ and $1 \leq k \leq n$, we define
\begin{align*}
e_p(i-1,k)&=\Vert r_{\sigma(\varphi_{i-1})} \circ \cdots \circ r_{\sigma(\varphi_1)}(t_k \mu) \Vert_p; \\
e_\infty(i-1,k)&=\Vert r_{\sigma(\varphi_{i-1})} \circ \cdots \circ r_{\sigma(\varphi_1)}(t_k \mu) \Vert_\infty.
\end{align*}
We note that this includes the case where $i=1$ so that
$e_p(0,k)=\Vert t_k \mu \Vert_p$ and
$e_\infty(0,k)=\Vert t_k \mu \Vert_\infty$.
When $i \leq 0$, the above definition should be understood as
$e_p(i-1,k)=0$ and
$e_\infty(i-1,k)=0$.

We use the following recursive representation of $\varphi_i$:
$$
\varphi_i=\Phi(r_{t_{i-1} \mu}(t_i\mu))=\Phi(r_{\sigma(\varphi_{i-1})} \circ r_{\sigma(\varphi_{i-2})} \circ \cdots \circ r_{\sigma(\varphi_1)}(t_i \mu)).
$$
This holds true because $\Phi \circ r_\mu=\Phi \circ r_{\mu'}$ if $\pi(\mu)=\pi(\mu')$ and
\begin{align*}
&\quad \pi(\sigma(\varphi_{i-1}) \ast \sigma(\varphi_{i-2}) \ast \cdots \ast \sigma(\varphi_1))\\
&=\pi(r_{t_{i-2} \mu}(t_{i-1}\mu) \ast r_{t_{i-3} \mu}(t_{i-2}\mu) \ast \cdots \ast (t_{1}\mu))
=\pi(t_{i-1} \mu).
\end{align*}
Here, Lemma \ref{cui} yields that
$$
\Vert \varphi_i \Vert_p \leq \frac{3e_p(i-1,i)}{2\sqrt{1-e_\infty(i-1,i)^2}}, 
$$
and as $e_\infty(i-1,i)$ $(1 \leq i \leq n)$
are uniformly bounded by a constant less than $1$ depending on $n$ and
$\Vert \mu \Vert_\infty$, there is a constant $C_1>0$ depending only on $\Vert \mu \Vert_\infty$ such that $\Vert \varphi_i \Vert_p \leq C_1 e_p(i-1,i)$.

We can obtain recursive inequalities for $e_p(i-1,k)$. First, we note that
\begin{align*}
e_p(i-1,k)&=\Vert r_{\sigma(\varphi_{i-1})} \circ \cdots \circ r_{\sigma(\varphi_1)}(t_k \mu)
\Vert_p\\
&=\Vert r_{\sigma(\varphi_{i-1})}(r_{\sigma(\varphi_{i-2})}
\circ \cdots \circ r_{\sigma(\varphi_1)}(t_k \mu))
-r_{\sigma(\varphi_{i-1})}(\sigma(\varphi_{i-1})) \Vert_p.
\end{align*}
Then, Proposition \ref{p-est} implies that there is a constant $C_0>0$ depending only on $\Vert \mu \Vert_\infty$ such that
\begin{align*}
&\quad \Vert r_{\sigma(\varphi_{i-1})}(r_{\sigma(\varphi_{i-2})}
\circ \cdots \circ r_{\sigma(\varphi_1)}(t_k \mu))
-r_{\sigma(\varphi_{i-1})}(\sigma(\varphi_{i-1})) \Vert_p \\
&\leq C_0 \Vert r_{\sigma(\varphi_{i-2})}\circ \cdots \circ r_{\sigma(\varphi_1)}(t_k \mu)-\sigma(\varphi_{i-1}) \Vert_p.
\end{align*}
Finally, the last $p$-norm is estimated as
\begin{align*}
&\quad \Vert r_{\sigma(\varphi_{i-2})}\circ \cdots \circ r_{\sigma(\varphi_1)}(t_k \mu)-\sigma(\varphi_{i-1}) \Vert_p \\
&\leq \Vert r_{\sigma(\varphi_{i-2})}\circ \cdots \circ r_{\sigma(\varphi_1)}(t_k \mu)\Vert_p +2\Vert \varphi_{i-1} \Vert_p\\
&\leq e_p(i-2,k)+2C_1 e_p(i-2,i-1).
\end{align*}
Hence, 
we have recursive inequalities 
$$
e_p(i-1,k) \leq Ce_p(i-2,k)+C e_p(i-2,i-1),
$$
where $C=\max\{C_0,2C_0C_1\}$.

From these inequalities, we can show that
$$
e_p(i-1,k) \leq C^{i-1}\{e_p(0,k)+\sum_{j=0}^{i-2} 2^j e_p(0,i-1-j)\}
$$
for $1 \leq i \leq n$ and $1 \leq k \leq n$. Indeed, this is valid for $i=k=1$.
Suppose that this is true for lower indices than $(i-1,k)$. Then,
\begin{align*}
e_p(i-1,k) &\leq Ce_p(i-2,k)+C e_p(i-2,i-1)\\
&\leq C \cdot C^{i-2}\{e_p(0,k)+\sum_{j=0}^{i-3} 2^j e_p(0,i-2-j)\}\\
&+ C \cdot C^{i-2}\{e_p(0,i-1)+\sum_{j=0}^{i-3} 2^j e_p(0,i-2-j)\},
\end{align*}
where the last term is equal to the desired one. Using 
$$
e_p(0,k)=\Vert t_k \mu \Vert_p=\frac{k}{n}\Vert \mu \Vert_p,
$$
we in particular obtain that
\begin{align*}
e_p(i-1,i) &\leq C^{i-1}\{e_p(0,i)+{\textstyle \sum_{j=0}^{i-2}}\ 2^j e_p(0,i-1-j)\}\\
&= C^{i-1}\left\{\frac{i}{n}+\sum_{j=0}^{i-2}\ 2^j \cdot \frac{i-1-j}{n}\right\} \Vert \mu \Vert_p\\
&= \frac{C^{i-1}(2^i-1)}{n} \Vert \mu \Vert_p.
\end{align*}

We will complete the estimate of the Weil--Petersson distance. We start with
$$
d_{WP}^p(\pi(0),\pi(\mu)) \leq \sum_{i=1}^n d_{WP}^p(\pi(t_{i-1} \mu),\pi(t_i \mu))=\sum_{i=1}^n d_{WP}^p(\pi(0),\pi(r_{t_{i-1} \mu}(t_i \mu)))
$$
by the invariance of $d_{WP}^p$ under the base point change.
Because $\varphi_i=\Phi(r_{t_{i-1} \mu}(t_i \mu))$ belongs to $U^\infty(\delta_0) \cap A^p(\D^*)$, 
as in Cui \cite{Cui} (see also
Corollary \ref{distancenorm} in the appendix for the precise estimate),
we have that 
$$
d_{WP}^p(\pi(0),\pi(r_{t_{i-1} \mu}(t_i \mu))) \leq 8 \Vert \varphi_i \Vert_p.
$$
In addition, $\Vert \varphi_i \Vert_p \leq C_1 e_p(i-1,i)$ as we have seen before. Hence,
\begin{align*}
d_{WP}^p(\pi(0),\pi(\mu)) &\leq 8 C_1 \sum_{i=1}^n e_p(i-1,i)\\
&\leq  \frac{8 C_1}{n}\sum_{i=1}^n C^{i-1}(2^i-1) \Vert \mu \Vert_p.
\end{align*}
This multiplier for $\Vert \mu \Vert_p$ is a constant depending only on $\Vert \mu \Vert_{\infty}$, and thus the proof is complete.
\qed
\medskip

\medskip
\noindent
{\it Proof of Theorem \ref{conjugate2}.}
For every $g \in G$, let $\varphi_g=\beta([g]) \in B(\D^*)$.
As $g \in \Diff_+^{1+\alpha}(\S1)$ and $p\alpha>1$,
we have $[g] \in T^p$ and hence $\varphi_g \in A^p(\D^*)$. By Lemma \ref{cui} and the assumption $k_p(g) \leq \varepsilon_p$, 
we see that $\Vert \varphi_g \Vert_p \leq 3\varepsilon_p/2$ for every $g \in G$. 
As $\Vert \varphi \Vert_\infty \leq c_p \Vert \varphi \Vert_p$ for $\varphi \in A^p(\D^*)$,
we have $\Vert \varphi_g \Vert_\infty \leq 3c_p \varepsilon_p/2$.
Hence, we can choose $\varepsilon_p >0$ so small that $\varphi_g \in U^\infty(1/2)$ and
then its image $\mu_g=\sigma(\varphi_g)$ of the Ahlfors--Weill section $\sigma:U^\infty(1/2) \to \Bel(\D)$ satisfies
$\Vert \mu_g \Vert_p \leq 3\varepsilon_p$ and $\Vert \mu_g \Vert_\infty \leq 3c_p \varepsilon_p$
for every $g \in G$.
This in particular implies that $G$ is a uniformly quasisymmetric group 
whose elements have quasiconformal extensions to $\D$ with a sufficiently small
dilatation bound $\kappa_\infty<1$ by choosing $\varepsilon_p >0$. 

Markovic \cite{Mar} proved that
a uniformly quasisymmetric group $G \subset \QS$ is conjugate into $\Mob(\S1)$
by a quasisymmetric homeomorphism $f_0 \in \QS$. 
This means that $\tau_0=[f_0] \in T$ is a fixed point of $G$.
Then, $G$ acts on the Banach space $B(\D^*)$ linear isometrically through the Bers embedding 
$\beta_{\tau_0}=\beta \circ R_{\tau_0}$.
Moreover, we can take $\tau_0$ sufficiently close to the origin $o=[\id]$
if $\kappa_\infty$ is sufficiently small. In particular, we may assume that
$\tau_0=[\nu]$ for $\nu \in \sigma(U^\infty(\delta_0))$.

The linear isometric action of $G$ on $B(\D^*)$ by $\beta_{\tau_0}$
also keeps the affine subspace $\varphi_0+A^p(\D^*) \subset B(\D^*)$ 
for $\varphi_0=\beta_{\tau_0}(o)=\Phi(\nu^{-1})$ invariant. 
Moreover,
the orbit of $\varphi_0$ under $G$, which is $\{\beta_{\tau_0}([g])\}_{g \in G}=\{\Phi(r_\nu(\mu_g))\}_{g \in G}$, 
is bounded with respect to the norm of $A^p(\D^*)$. These facts can be verified
by Lemma \ref{c-y} and Proposition \ref{p-est} (see also the remark after this proposition)
as follows. The combination of these claims yields that if $r_\nu(\mu_2),\ \nu \in \sigma(U^\infty(\delta_0))$
then
\begin{align*}
\Vert \Phi(r_\nu(\mu_1))-\Phi(r_\nu(\mu_2)) \Vert_p 
&\leq C\Vert \mu_1-\mu_2\Vert_p, 
\end{align*}
where $C>0$ is a constant depending only on $\Vert \mu_1\Vert_\infty$, $\Vert \mu_2 \Vert_\infty$ and $\Vert \nu \Vert_\infty$.
We apply this inequality for $\mu_1=\mu_g$ $(g \in G)$, $\mu_2=0$ and $\nu$ as above.
Then, 
$$
\Vert \Phi(r_\nu(\mu_g))-\varphi_0 \Vert_p \leq C \Vert \mu_g \Vert_p \leq 3\varepsilon_p C
$$
for every $g \in G$.

This gives a fixed point of $G$ in $\varphi_0+A^p(\D^*)$ because the Banach space $A^p(\D^*)$ is 
uniformly convex; any bounded orbit has a unique circumcenter also in this case.
See \cite{Mat5} for a survey of this property.
Furthermore, we choose $\varepsilon_p$ so small that the closed ball in 
$\varphi_0+A^p(\D^*)$ of center at $\varphi_0$ and radius $3\varepsilon_p C$ is contained in
the open set
$$
\beta_{\tau_0}(T^p) = \beta(T) \cap (\varphi_0+A^p(\D^*)).
$$
This equality is mentioned in the remark after Proposition \ref{p-est}.
Then, the fixed point of $G$, which is the circumcenter of the orbit, is also in $\beta_{\tau_0}(T^p)$.
Having the new fixed point $\tau=[f] \in T^p$, 
we perform the same argument as 
before. Namely, we know that $[f] \in T_0^\alpha$ by Corollary \ref{cor-rigidity}.
Hence, $f \in \Diff_+^{1+\alpha}(\S1)$ satisfies $fGf^{-1} \subset \Mob(\S1)$.
\qed
\medskip

If $(T^p,d_{WP}^p)$ possesses the fixed point property such that every isometry group with a bounded orbit 
has a fixed point in it, then the statement of Theorem \ref{conjugate2} can be improved 
as in Theorem \ref{conjugate}. 
We expect that $(T^p,d_{WP}^p)$ satisfies certain uniform convexity and hence
the fixed point property. See \cite{Mat5} for the relation of these conditions.
In the next appendix, we will consider the $p$-Weil--Petersson metric towards this problem.

\section{Appendix: The $p$-Weil--Petersson metric}\label{8}
We prove basic properties of the $p$-Weil--Petersson metric $d_{WP}^p$
such as the completeness of the distance induced by the metric and
the continuity of the metric as the base point varies.
To this end, 
we first compare the $p$-Weil--Petersson distance $d_{WP}^p(\cdot,\cdot)$ on $T^p$
with the $p$-integrable norm $\Vert \cdot \Vert_p$ 
in a small open ball of $A^p(\D^*)$ centered at the origin.

Let $U^p(r) \subset A^p(\D^*)$ and $U^\infty(r) \subset B(\D^*)$ denote the open balls of radius $r$ centered at the origin.
We set $\delta_p=\delta_0/c_p$, where $c_p$ is the constant satisfying the condition $\Vert \varphi \Vert_\infty \leq c_p \Vert \varphi \Vert_p$
as in Proposition \ref{inclusion} and 
$\delta_0 \leq 1/4$ is the constant as in Proposition \ref{jacobian}. Then, $U^p(\delta_p) \subset U^\infty(\delta_0) \subset \beta(T)$.
We note that $U^\infty(\delta_0)$ is properly contained in the domain $U^\infty(1/2)$ of the Ahlfors--Weill section 
$\sigma:U^\infty(1/2) \to \Bel(\D)$, and hence the Teich\-m\"ul\-ler distance $d_T$ on $T$ and 
the supremum norm $\Vert \cdot \Vert_\infty$ of $B(\D^*)$
are comparable there (see Lehto \cite[Section III.4.2]{Leh}).

For the base point change map $R_\tau$ for $\tau \in T^p$, which is a biholomorphic 
and isometric automorphism of $(T^p,d_{WP}^p)$ sending $\tau$ to $o=[\id]$,
we consider its conjugate by the Bers embedding $\beta:T^p \to \beta(T) \cap A^p(\D^*)$:
$$
R^*_{\beta(\tau)}=\beta \circ R_\tau \circ \beta^{-1}.
$$
For each $\varphi \in \beta(T) \cap A^p(\D^*)$, $R^*_{\varphi}$ is a biholomorphic automorphism of $\beta(T) \cap A^p(\D^*)$
sending $\varphi$ to $0$.
The derivative $d_\psi R^*_{\varphi}:A^p(\D^*) \to A^p(\D^*)$ at any $\psi \in \beta(T) \cap A^p(\D^*)$ is a bounded linear operator.
The following result and its corollary are similarly obtained by Cui \cite[Theorem 4]{Cui} for $p=2$.

\begin{theorem}\label{derivativenorm}
The operator norm of the derivative $d_{\varphi} R^*_{\varphi}$ at $\varphi$ 
and that of $d_{0} (R^*_{\varphi})^{-1}$ at $0$ satisfy
$\Vert d_{\varphi} R^*_{\varphi} \Vert \leq 8$ and $\Vert d_{0} (R^*_{\varphi})^{-1} \Vert \leq 64$
for every $\varphi \in U^\infty(\delta_0) \cap A^p(\D^*)$.
\end{theorem}

\begin{proof}
First, we consider $R^*_{\varphi}$ and estimate of the norm of its derivative at $\varphi$ from above.
We decompose $R^*_{\varphi}$ into
$R^*_{\varphi}=\Phi \circ r_{\nu} \circ \sigma$,
where $\nu=\sigma(\varphi) \in \Ael^p(\D)$. Then,
$$
d_{\varphi} R^*_{\varphi}=d_0 \Phi \circ d_{\nu} r_{\nu} \circ d_{\varphi}\sigma.
$$
Here, the Ahlfors--Weill section $\sigma$ is linear with $\Vert d\sigma \Vert=2$ and the remark after
Lemma \ref{cui} implies that
$\Vert d_0 \Phi \Vert \leq 3/2$. On the contrary, because
$$
(d_{\nu} r_{\nu})(\lambda)(w)=\frac{\lambda(z)}{1-|\nu(z)|^2} \frac{\partial f^\nu(z)}{\overline{\partial f^\nu(z)}} \qquad (w=f^\nu(z))
$$
for a tangent vector $\lambda$ of $\Ael^p(\D)$ at $\nu$, we see that
\begin{align*}
\Vert (d_{\nu} r_{\nu})(\lambda) \Vert_p &\leq 
\frac{1}{1-\Vert \nu \Vert_\infty^2} \left(\int_{\D} |\lambda((f^\nu)^{-1}(w))|^p \rho_{\D}^2(w)dudv\right)^{1/p}\\
&=\frac{1}{1-\Vert \nu \Vert_\infty^2} \left(\int_{\D} |\lambda(z)|^p \rho_{\D}^2(f^{\nu}(z)) J_{f^{\nu}}(z)dxdy \right)^{1/p}\\
&\leq\frac{2^{1/p}}{1-(2\delta_0)^2} \left(\int_{\D} |\lambda(z)|^p \rho_{\D}^2(z) dxdy \right)^{1/p},
\end{align*}
where the last inequality stems from Proposition \ref{jacobian} and
$\Vert \nu \Vert_\infty=2 \Vert \varphi \Vert_\infty \leq 2 \delta_0$. As $\delta_0 \leq 1/4$,
this shows that $\Vert d_{\nu} r_{\nu} \Vert \leq 8/3$. 
Consequently, we have that
$$
\Vert d_{\varphi} R^*_{\varphi}\Vert \leq \Vert d_0 \Phi \Vert \cdot \Vert d_{\nu} r_{\nu} \Vert \cdot \Vert d_{\varphi}\sigma \Vert
\leq 8.
$$

Next, we consider $(R^*_{\varphi})^{-1}$ and estimate of the norm of its derivative at $0$ from above.
We decompose $(R^*_{\varphi})^{-1}$ into
$(R^*_{\varphi})^{-1}=\Phi \circ r_{\nu^{-1}} \circ \sigma$.
Then,
$$
d_{0} (R^*_{\varphi})^{-1}=d_{\nu} \Phi \circ d_{0} r_{\nu^{-1}} \circ d_{0}\sigma.
$$
As before $\Vert d\sigma \Vert=2$. Lemma \ref{c-y} implies that
$$
\Vert d_{\nu} \Phi(\widetilde \lambda) \Vert_p \leq \frac{12}{1-\Vert \nu \Vert_\infty^2} \Vert \widetilde \lambda \Vert_p
$$
for a tangent vector $\widetilde \lambda$ of $\Ael(\D)$ at $\nu$.
Hence, $\Vert d_{\nu} \Phi \Vert \leq 12/(1-\Vert \nu \Vert_\infty^2) \leq 16$. 
For the derivative $d_{0} r_{\nu^{-1}}$, we know that
\begin{align*}
d_{0} r_{\nu^{-1}}(\lambda_*)(w)&=\lambda_*(z)(1-|\nu^{-1}(z)|^2) 
\frac{\partial (f^{\nu})^{-1}(z)}{\overline{\partial (f^{\nu})^{-1}(z)}}\\ 
&=\lambda_*(f^\nu(w))(1-|\nu(w)|^2) 
\frac{\overline{\partial f^{\nu}(w)}}{\partial f^{\nu}(w)} \qquad (w=(f^{\nu})^{-1}(z))
\end{align*}
for a tangent vector $\lambda_*$ of $\Ael^p(\D)$ at $\nu^{-1}$. Then,
\begin{align*}
\Vert (d_{0} r_{\nu^{-1}})(\lambda_*) \Vert_p &\leq 
\left(\int_{\D} |\lambda_*(f^{\nu}(w))|^p \rho_{\D}^2(w)dudv\right)^{1/p}\\
&=\left(\int_{\D} |\lambda_*(z)|^p \rho_{\D}^2(f^{\nu}(z)) J_{f^{\nu}}(z)dxdy \right)^{1/p}\\
&\leq 2^{1/p} 
\left(\int_{\D} |\lambda_*(z)|^p \rho_{\D}^2(z) dxdy \right)^{1/p}.
\end{align*}
Hence, we have $\Vert d_{0} r_{\nu^{-1}} \Vert \leq 2$.
These estimates together conclude that
$$
\Vert d_{0} (R^*_{\varphi})^{-1} \Vert \leq \Vert d_{\nu} \Phi \Vert \cdot \Vert d_{0} r_{\nu^{-1}} \Vert \cdot \Vert d_{0}\sigma \Vert
\leq 64,
$$
which completes the proof.
\end{proof}

We have an immediate consequence from this theorem, which shows the bi-Lipschitz continuity between
$\Vert \cdot \Vert_p$ and $d_{WP}^p$ near the origin. 

\begin{corollary}\label{distancenorm}
$(1)$ Any $\varphi_0, \varphi_1 \in U^{p}(\delta_p/3) \subset U^{\infty}(\delta_0) \cap A^p(\D^*)$ satisfy
$$
\frac{1}{64} \Vert \varphi_1-\varphi_0 \Vert_p \leq d_{WP}^p(\beta^{-1}(\varphi_1),\beta^{-1}(\varphi_0)) 
\leq 8 \Vert \varphi_1-\varphi_0 \Vert_p.
$$
The upper estimate is still valid for any $\varphi_0, \varphi_1 \in U^{\infty}(\delta_0) \cap A^p(\D^*)$.
$(2)$ If $\tau \in T^p$ satisfies $d_{WP}^p(\tau,o)<c\delta_p/64$ for some
$c \in (0,1]$, then $\beta(\tau) \in U^p(c\delta_p)$.
\end{corollary}

\begin{proof}
(1) For the upper estimate, we choose the segment 
$\gamma_0=\{t\varphi_1+(1-t)\varphi_0\}_{t \in [0,1]}$ in $U^{\infty}(\delta_0) \cap A^p(\D^*)$
connecting $\varphi_0$ and $\varphi_1$. 
Then, the $p$-Weil--Petersson length $\ell_{WP}^p(\gamma_0)$ of $\gamma_0$ is given by
$$
\ell_{WP}^p(\gamma_0)=\int_0^1 \Vert (d_{t\varphi_1+(1-t)\varphi_0}R^*_{t\varphi_1+(1-t)\varphi_0})
(\varphi_1-\varphi_0) \Vert_p\, dt.
$$
The continuity of $\Vert d_{\bullet} R^*_{\bullet}(\varphi) \Vert_p$ will be seen in Theorem \ref{continuous}.
As Theorem \ref{derivativenorm} yields that
$$
\Vert (d_{t\varphi_1+(1-t)\varphi_0}R^*_{t\varphi_1+(1-t)\varphi_0})
(\varphi_1-\varphi_0) \Vert_p \leq 8 \Vert \varphi_1-\varphi_0 \Vert_p,
$$ 
we see that $\ell_{WP}^p(\gamma_0) \leq 8 \Vert \varphi_1-\varphi_0 \Vert_p$. This shows 
the upper estimate.

For the lower estimate, we choose a smooth arc $\gamma$ in $\beta(T^p)$ connecting $\varphi_0$ and $\varphi_1$
whose length $\ell_{WP}^p(\gamma)$ is
arbitrarily close to $d_{WP}^p(\beta^{-1}(\varphi_0),\beta^{-1}(\varphi_1))$. We give an arc length parameter $s$ for $\gamma$
with respect to the norm $\Vert \cdot \Vert_p$; its parametrization is
$\{\gamma(s)\}_{s \in [0,S]}$ for $S \geq \Vert \varphi_1- \varphi_0\Vert_p$, where $\gamma(0)=\varphi_0$ and
$\gamma(S)=\varphi_1$.
We may assume that $\gamma$ is contained in $U^p(\delta_p)$, for otherwise, we replace $\gamma$ with a sub-arc 
$\{\gamma(s)\}_{s \in [0,S']}$ in $U^p(\delta_p)$ that still holds $S' \geq \Vert \varphi_1- \varphi_0\Vert_p$.
This is possible because $\Vert \varphi_1- \varphi_0\Vert_p<2\delta_p/3$ and $\Vert \varphi_0 \Vert_p<\delta_p/3$.
Then,
$$
\ell_{WP}^p(\gamma)=\int_0^S \Vert (d_{\gamma(s)}R^*_{\gamma(s)})(\dot\gamma(s)) \Vert_p\, ds.
$$
Here, Theorem \ref{derivativenorm} implies that the integrand is bounded from below by $1/64$. 
Hence, $\ell_{WP}^p(\gamma) \geq S/64 \geq \Vert \varphi_1-\varphi_0 \Vert_p/64$.
As $\ell_{WP}^p(\gamma)$ can be arbitrarily close to 
the distance $d_{WP}^p(\beta^{-1}(\varphi_0),\beta^{-1}(\varphi_1))$, we 
obtain the lower estimate. 

(2) We consider a smooth arc $\gamma$ in $\beta(T^p)$ connecting $\beta(\tau)$ and $0$
whose length $\ell_{WP}^p(\gamma)$ is
arbitrarily close to $d_{WP}^p(\tau,o)$. As before, we give the arc length parametrization $\gamma(s)$
with respect to the norm $\Vert \cdot \Vert_p$.
If $\Vert \beta(\tau) \Vert_p \geq c\delta_p$ for $c \in (0,1)$, there exists some $S > 0$ such that
$\{\gamma(s)\}_{s \in [0,S]}$ is contained in $U^p(\delta_p)$ and $S \geq c\delta_p$. Then,
$\ell_{WP}^p(\gamma) \geq S/64 \geq c\delta_p/64$ by Theorem \ref{derivativenorm}.
This implies that $d_{WP}^p(\tau,o) \geq c\delta_p/64$.
\end{proof}

The completeness of the distance then follows from this corollary. 
This was also obtained by Cui \cite[Theorem 5]{Cui} for $p=2$.

\begin{theorem}\label{complete}
The $p$-Weil--Petersson distance $d_{WP}^p(\cdot,\cdot)$ is complete on $T^p$.
\end{theorem}

\begin{proof}
We consider any Cauchy sequence in $(T^p,d_{WP}^p)$. It suffices to consider its tail
whose diameter can be arbitrary small. As the isometric automorphism group acts transitively on $T^p$,
we may assume that the tail of the Cauchy sequence is contained in the open ball of radius $\delta_p/192$ centered at the origin.
Corollary \ref{distancenorm} implies that it is in $\beta^{-1}(U^p(\delta_p/3))$ and
its Bers embedding is a convergent sequence 
with respect to the norm $\Vert \cdot \Vert_p$. Hence, the Cauchy sequence also converges
with respect to $d_{WP}^p$.
\end{proof}

We also obtain the following result, which is a counterpart to
Theorem \ref{WPdistance}. 
We note here that the Teich\-m\"ul\-ler distance $d_T$ is
bounded by the $p$-Weil--Petersson distance $d_{WP}^p$ multiplied by a certain constant $c'_p>0$ depending only on $p$, that is,
$d_T(\tau_1,\tau_2) \leq c'_p d_{WP}^p(\tau_1,\tau_2)$ for any $\tau_1, \tau_2 \in T^p$.
See \cite[Proposition 6.10]{Mat3}. It is asked whether the constant $C$ below can be taken so that it depends only on $d_T(o,\tau)$.

\begin{proposition}\label{p-norm}
For every $\tau \in T^p$,
there is $\mu \in \Ael^p(\D)$ with $\tau=\pi(\mu)$ such that 
$$
\Vert \mu \Vert_p \leq Cd_{WP}^p(o,\tau)
$$ 
for a constant $C>0$ depending only on $d_{WP}^p(o,\tau)$.
\end{proposition}

\begin{proof}
We choose a finite sequence of points $\{\tau_i\}_{i=0}^n \subset T^p$ so that
$\tau_0=o$, $\tau_n=\tau$, and $d_{WP}^p(\tau_{i-1},\tau_i) < \delta_p/128$ for $1 \leq i \leq n$.
In addition, we can choose the number $n \geq 1$ so that
$$
(n-1) \delta_p/128 \leq d_{WP}^p(o,\tau)< n \delta_p/128
$$
is satisfied.

For $\tau_1$ with $d_{WP}^p(o,\tau_1) < \delta_p/128$, we see that
$\Vert \beta(\tau_1) \Vert_p < \delta_p/2$ by Corollary \ref{distancenorm}. 
We set $\mu_1=\sigma(\beta(\tau_1)) \in \Ael^p(\D)$, which satisfies
$\pi(\mu_1)=\tau_1$. 
Then, $\Vert \mu_1 \Vert_p < \delta_p$ and $\Vert \mu_1 \Vert_\infty < \delta_0 \leq 1/2$.
This also implies that in the case of $n=1$, we obtain the required estimate with $C=128$.
Hence, we may assume that $n \geq 2$ hereafter.

For $\tau_2$, we consider $\tau'_2=R_{\tau_1}(\tau_2) \in T^p$. As $d_{WP}^p(o,\tau'_2)=d_{WP}^p(\tau_1,\tau_2) < \delta_p/128$,
we see that
$\Vert \beta(\tau'_2) \Vert_p < \delta_p/2$. Let $\mu'_2=\sigma(\beta(\tau'_2)) \in \Ael^p(\D)$, which satisfies
$\pi(\mu'_2)=\tau'_2$, $\Vert \mu'_2 \Vert_p < \delta_p$, and $\Vert \mu'_2 \Vert_\infty < \delta_0$. 
Then, $\mu_2=\mu'_2 \ast \mu_1$ satisfies $\pi(\mu_2)=\tau_2$ and $\Vert \mu_2 \Vert_\infty \leq 4/5$. 
Moreover, by Proposition \ref{p-est} and the remark after that, we have that
$$
\Vert \mu_2-\mu_1 \Vert_p=\Vert r_{\mu_1^{-1}}(\mu'_2)-r_{\mu_1^{-1}}(0) \Vert_p
\leq C_1 \Vert \mu'_2 \Vert_p \leq C_1 \delta_p,
$$
where $C_1>0$ is a constant depending only on $\Vert \mu^{-1}_1 \Vert_\infty=\Vert \mu_1 \Vert_\infty$.

For $\tau_3,\ldots,\tau_n$, we repeat the same argument. Inductively, we define $\mu'_i$ and $\mu_i$ $(i=3,\ldots,n)$
similarly, which satisfy $\Vert \mu'_i \Vert_p < \delta_p$ and 
$\Vert \mu_i \Vert_\infty \leq (3^i-1)/(3^i+1)$. 
Then, we also have
$$
\Vert \mu_i-\mu_{i-1} \Vert_p \leq C_{i-1}\delta_p
$$
for a constant $C_{i-1}>0$ depending only on $\Vert \mu_{i-1} \Vert_\infty$.

Therefore, by summing up these estimates, we conclude that
$$
\Vert \mu_n \Vert_p \leq \Vert \mu_1 \Vert_p +\Vert \mu_2-\mu_1 \Vert_p + \cdots +\Vert \mu_n-\mu_{n-1} \Vert_p
\leq C' \delta_p n
$$
for $\mu_n \in \Ael(\D)$ with $\pi(\mu_n)=\tau_n=\tau$, where $C'>0$ is a constant depending only on $n$.
As $n \geq 2$ satisfies $(n-1) \delta_p/128 \leq d_{WP}^p(o,\tau)$, this yields the required inequality.
\end{proof}

Finally, we show the continuity of the metric $d_{WP}^p$, by which
it can be accepted as a Finsler metric. We note that the differentiability of the metric is also verified
by further arguments, which has been proved by Yanagishita \cite{Yan-new}.

\begin{theorem}\label{continuous}
The $p$-Weil--Petersson metric $d_{WP}^p$ is continuous on $T^p$.
\end{theorem}

\begin{proof}
It suffices to show that for each tangent vector $\psi \in A^p(\D^*)$, $\Vert d_{\varphi} R^*_{\varphi}(\psi) -\psi\Vert_p$
converge to $0$ as $\varphi$ tend to $0$ in $A^p(\D^*)$. Because
$$
d_{\varphi} R^*_{\varphi}=d_0 \Phi \circ d_{\nu} r_{\nu} \circ d_{\varphi}\sigma
$$
for $\nu=\sigma(\varphi)$, we have
$$
d_{\varphi} R^*_{\varphi}(\psi)(z)=-\frac{6}{\pi}\int_{\D} \frac{(d_{\nu} r_{\nu})(\lambda)(w)}{(w-z)^4} dudv \qquad (z \in \D^*)
$$
for $\lambda=d \sigma(\psi)$ by the remark after Lemma \ref{cui}. Hence,
\begin{align*}
& \Vert d_{\varphi} R^*_{\varphi}(\psi) \Vert_p^p
=\int_{\D^*} \left|\frac{6}{\pi}\int_{\D} \frac{(d_{\nu} r_{\nu})(\lambda)(w)}{(w-z)^4} dudv \right|^p \rho_{\D^*}^{2-2p}(z) dxdy;\\
& \Vert \psi \Vert_p^p
=\int_{\D^*} \left|\frac{6}{\pi}\int_{\D} \frac{\lambda(w)}{(w-z)^4} dudv \right|^p \rho_{\D^*}^{2-2p}(z) dxdy.
\end{align*}

Here, as we have seen in the proof of Theorem \ref{derivativenorm},
$$
(d_{\nu} r_{\nu})(\lambda)(w)=\frac{\lambda(\zeta)}{1-|\nu(\zeta)|^2} \frac{\partial f^\nu(\zeta)}{\overline{\partial f^\nu(\zeta)}}  
$$
for $w=f^\nu(\zeta)$, which is $p$-integrable. 
Moreover, 
it converges to $\lambda(w)$ as $\varphi \to 0$. In particular, if 
we restrict the inner integrals in the above formulae to a smaller disk $|w| \leq r$ for any $r \in (0,1)$, 
the dominated convergence theorem can be 
applied to that part to see the convergence 
$$
\left(\int_{\D^*} \left|\frac{6}{\pi}\int_{_{|w| \leq r}} \frac{(d_{\nu} r_{\nu})(\lambda)(w)-\lambda(w)}{(w-z)^4} dudv \right|^p \rho_{\D^*}^{2-2p}(z) dxdy \right)^{1/p} \to 0 \qquad (\varphi \to 0).
$$

Next, we consider a uniform estimate of the $p$-integral of $(d_{\nu} r_{\nu})(\lambda)(w)$ on $r<|w|<1$.
By change of variables $w=f^\nu(\zeta)$, we have
$$
\Vert (d_{\nu} r_{\nu})(\lambda) \Vert^p_p
=\int_{\D}\left |\frac{\lambda(\zeta)}{1-|\nu(\zeta)|^2} \right|^p \rho_{\D}^2(f^\nu(\zeta))J_{f^\nu}(\zeta) d\xi d\eta.
$$
Proposition \ref{jacobian} implies that there is $\delta \in (0,1/2)$ such that
$$
\left |\frac{\lambda(\zeta)}{1-|\nu(\zeta)|^2} \right|^p\rho_{\D}^2(f^\nu(\zeta))J_{f^\nu}(\zeta)
\leq 2|\lambda(\zeta)|^p \rho_{\mathbb D}^2(\zeta)
$$
for every $\zeta \in \D$ and for every $\varphi \in U^\infty(\delta)$.
Because $f^\nu$ converge to $\id$ uniformly on $\D$ as $\varphi  \to 0$
and $\lambda$ is $p$-integrable, we see that for every $\widetilde \varepsilon>0$, there is some $r \in (0,1)$ such that 
$$
\left(\int_{r<|w|<1} |(d_{\nu} r_{\nu})(\lambda)(w)|^p \rho_{\D}^2(w)dudv \right)^{1/p} \leq \widetilde \varepsilon
$$
for every $\varphi \in U^\infty(\delta)$ by replacing $\delta$ with a smaller constant if necessary.
We note that in the case of $\varphi=0$, this also gives the estimate for $\lambda(w)$.

By Lemma \ref{cui} with the remark on the fact $\Vert d_0\Phi \Vert \leq 3/2$, we obtain that
$$
\left(\int_{\D^*} \left|\frac{6}{\pi}\int_{r<|w|<1} \frac{(d_{\nu} r_{\nu})(\lambda)(w)}{(w-z)^4} dudv \right|^p \rho_{\D^*}^{2-2p}(z) dxdy \right)^{1/p} \leq \frac{3}{2} \widetilde \varepsilon
$$
for every $\varphi \in U^\infty(\delta)$. Then, together with the consequence from the dominated convergence theorem, 
we can conclude that
$$
\limsup_{\varphi \to 0} \Vert d_{\varphi} R^*_{\varphi}(\psi) -\psi\Vert_p \leq 3 \widetilde \varepsilon.
$$
As $\widetilde \varepsilon>0$ can be taken arbitrarily small, the proof is complete.
\end{proof}

\medskip
\noindent
{\it Acknowledgement.}
The author would like to thank the referee for his/her careful reading of the manuscript.

\end{document}